\documentclass[11pt]{article}
\usepackage{amsmath,amssymb,amsthm,enumerate,graphicx,bbm,appendix,booktabs,enumitem,verbatim,subfigure}
\usepackage{multirow}
\usepackage{array}
\usepackage{float}
\usepackage[dvipsnames,usenames]{color}

\def\fskip#1{}

\newtheorem{theorem}{Theorem }

\newtheorem{algorithm}{Algorithm}

\newtheorem{lemma}[theorem]{Lemma}
\newtheorem{definition}{Definition}[section]

\newtheorem{proposition}[theorem]{Proposition}

\newtheorem{assumption}{Assumption}

\newcommand{\Real}{\ensuremath{\mathbb{R}}}

\def\be{\begin{enumerate}}
\def\ee{\end{enumerate}}

\def\st{\mbox{subject to}}
\newcommand{\pmat}[1]{\begin{pmatrix} #1 \end{pmatrix}}

		\def\bkE{{\rm I\kern-.17em E}}
		\def\bk1{{\rm 1\kern-.17em l}}
		\def\bkD{{\rm I\kern-.17em D}}
		\def\bkR{{\rm I\kern-.17em R}}
		\def\bkP{{\rm I\kern-.17em P}}
		\def\bkY{{\bf \kern-.17em Y}}
		\def\bkZ{{\bf \kern-.17em Z}}

\def\xbar{\bar{x}}

\def\R{\mathbb{R}}

\def\argmin{\mathop{\rm argmin}}

\title{On the resolution of misspecified convex optimization
	and monotone variational inequality problems}
\author{Hesam~Ahmadi and Uday~V.~Shanbhag\thanks{Ahmadi and Shanbhag are
with the Department of Industrial and Manufacturing
Engineering, respectively at the Pennsylvania State University,
	University Park, PA-16803. They are reachable at ({\tt
		ahmadi.hesam@gmail.com,udaybag@psu.edu}) and their research has been
partially funded by NSF awards no. 1246887 (CAREER) and 1400217.}}
\begin{document}
\maketitle
\thispagestyle{empty}
\pagestyle{empty}
\begin{abstract}
We consider a  misspecified optimization problem that requires
minimizing a function $f(x;\theta^*)$ over {a closed and convex set} $X$
where $\theta^*$ is an unknown vector of parameters that {may be} learnt
by a parallel learning process.  In this context, we examine the
development of coupled schemes that generate iterates $(x_k,\theta_k)$
as  $k \to \infty$, then $x_k \to x^*$, a minimizer of $f(x;\theta^*)$
over $X$ and $\theta_k \to \theta^*$. In the first part of the paper, we
consider the solution of problems where $f$ is either smooth or
nonsmooth. In smooth strongly convex regimes, we demonstrate that such
schemes lead to a quantifiable degradation of the standard linear
convergence rate. When strong convexity assumptions are weakened, it can
be shown that the convergence in function values sees a modification in
the convergence rate of ${\cal O}(1/K)$ by an additive factor
$\|\theta_0-\theta^*\|{\cal
O}(q_g^K + 1/K)$ where $\|\theta_0-\theta^*\|$ represents the initial
misspecification in $\theta^*$ and $q_g$ denotes the contractive factor associated
with the learning process. In both convex and strongly convex regimes,
diminishing steplength schemes are also provided and are less reliant on
the knowledge of problem parameters. Finally, we present an
averaging-based subgradient scheme and show that the optimal constant
steplength leads to a modification in the rate by
$\|\theta_0-\theta^*\|{\cal O}(q_g^K +
1/K)$, implying no effect on the standard rate of ${\cal
	O}(1/\sqrt{K})$. In the second part of the paper, we consider the
	solution of misspecified monotone variational inequality problems,
	motivated by the need to contend with more general equilibrium
	problems as well as the possibility of misspecification in the
	constraints. In this context, we first present a constant steplength
	misspecified extragradient scheme and prove its asymptotic
	convergence.  This scheme is reliant on problem parameters (such as
			Lipschitz constants) and leads us to present a misspecified
	variant of iterative Tikhonov regularization. Numerics support the
	asymptotic and rate statements  with one important observation: it
	appears that the rate bound derived for strongly convex problems
	appears to be slack in that the standard linear rate is again
	observed, despite the theoretical prediction that learning leads to
	degradation.
	\end{abstract}

\section{Introduction}
Traditionally, the field of deterministic optimization has focused on
the problem of minimizing a function $f(x)$ over a prescribed set $X$
and it is {generally} assumed that the decision maker has complete
knowledge of both the function $f$ and the set $X$
(cf.~\cite{gill81practical,bertsekas99nonlinear}). {In many settings,
problem data may be uncertain, severely limiting the applicability of
deterministic methods. Initiated through the research by
Dantzig~\cite{dantzig55linear} and Beale~\cite{beale55minimizing},
stochastic programming has represented a popular avenue for addressing
risk-neutral as well as risk-averse static and adaptive (recourse-based)
decision-making problems~\cite{birge97introduction,Shapiro09lecturesSA}
in developing both static as well as adaptive (recourse-based) models.
An alternative approach has found merit by obviating the need for
distributional information and instead focuses on obtaining solutions
that are {\em robust} to parametric uncertainty over a prescribed
(uncertainty) set~\cite{BEN:09,Bertsimas:2011:TAR:2079012.2079015}. In
either instance, a subset of parameters is natively uncertain. Our focus
is on a class of problems in which the vector parameters is $\theta^*$,
   a fixed {\em but unknown} vector, that may be learnt through a
	   related but distinct learning process.  We provide a clearer
	   understanding of our problem of interest by considering a
	   motivating problem.\\

\noindent {\bf Data-driven stochastic optimization:} Standard models for
stochastic optimization have required the solution of the following
problem~\cite{birge97introduction,Shapiro09lecturesSA}:
\begin{align}
\tag{StochOpt$(\theta^*)$} \min_{x \in X} \, \mathbb{E}_{\theta^*}[
f(x;\xi(\omega))],
\end{align}
where $X \subseteq \Real^n$, $f:X \times \Real^d \to \Real$, $\xi:\Omega
\to \Real^d$ is a $d-$dimensional random variable, and $(\Omega,
		{\mathcal F}, \mathbb{P}_{\theta^*})$ denotes the probability
space. Note that $\theta^*$ represents the parameters of the
distribution $\mathbb{P}$. Unfortunately, a key shortcoming in the use
of standard models necessitates knowledge of $\theta^*$, often a
stringent requirement. Instead, suppose $\theta^*$ may be learnt by a
suitably defined  maximum likelihood estimation (MLE)
problem~\cite{hastie01elements}, captured by a
metric $g(\theta)$, and formally defined as follows:
\begin{align}
\tag{MLE} \min_{\theta \in \Theta} \, g(\theta).
\end{align}
Generally, in most practically occurring problems, the MLE problem is
often massive and one avenue lies in generating sequences
$\{(x_k,\theta_k)\}$ such that $x_k$ is an approximate solution of
(StochOpt$(\theta_k)$).\\

A range of other problems can be cast in a similar regime. For instance,
in traffic equilibrium problems~\cite{facchinei02finite}, a common
assumption is that the demand pattern and the travel times are known
vectors, assumptions that are often hard to justify in practice.
Similarly, in a range of production planning problems, it is routinely
assumed that cost and demand information is accurately available when in
fact, it needs to be empirically estimated. Consequently, one approach
lies in conducting such an estimation through a parallel learning
process.  Yet another problem that can be cast under this umbrella is
the well studied multi-armed bandit problem~\cite{gittins89}. In such a
problem, a gambler is faced with a choosing from a collection of slot
machines at every step without a prior knowledge of the average reward
distribution.  If one could view the learning problem associated with
the reward distributions as a parallel estimation problem, this may be
one avenue towards developing algorithmic techniques. Motivated by this new set of decision-making problems, we consider the {\em static misspecified convex optimization
problem} $(\cal C(\theta^*))$, defined as follows:
\begin{align}
\tag{$\cal C(\theta^*)$} \min_{x \in X} \, f(x,\theta^*), \end{align} where
$x \in \Real^n$, $f: X \times \Theta \to \Real$ is a convex function in
$x$ for every $\theta \in \Theta \subseteq \Real^m$. Suppose $\theta^*$
denotes the solution to a convex learning problem denoted by $(\cal L)$:
\begin{align}
\tag{${\cal L}$} \min_{\theta \in \Theta} \ g(\theta), \end{align}
where $g: \Real^m \to \Real$ is a convex function in $\theta$ and is
defined on a closed and convex set $\Theta.$ Consequently, we consider
gradient methods in which sequences
$\{x_k\}$ and $\{\theta_k\}$ may be generated with the goal that
$$ \lim_{k\to \infty} x_k = x^* \in X^* \mbox{ and } \lim_{k\to \infty} \theta_k
= \theta^*, \mbox{ where } X^* \triangleq \displaystyle{ \argmin_{x \in X}
	f(x;\theta^*)}.$$
It should be noted that the second author has examined the counterpart
of such problems in stochastic regimes where stochastic approximation
schemes are employed~\cite{jiang13solution}. 

\subsection{Alternate and related avenues} Given that the focus lies on solving
$({\cal C}(\theta^*))$ and $(\cal L)$ simultaneously, at least three
approaches assume relevance and are described next.
\paragraph{A sequential approach:}
A natural question is whether this problem could indeed be solved in a
sequential fashion. For instance, one approach could be to compute
$\theta^*$ in the first stage and subsequently solve ($\cal
		C(\theta^*)$). Yet such an avenue is complicated by several
challenges: (i) First, the problem ($\cal L$) is often of a large or
massive scale and accurate/exact solutions of this problem are needed in
finite time to utilize this approach. However, the claim that finite
termination schemes are available is a strong one. In fact, even in the
rare instance when this requirement is met, the number of steps might be
far too large in practice.  Consequently, such an approach leads to
obtaining an approximation of $\theta^*$, given by a vector $\hat
\theta$ and {\bf cannot} provide asymptotically accurate solutions. (ii)
{Second, if the process is terminated prior to the commencing with the
computation of $x^*$, then the resulting computational effort would be
wasted in that we have no guarantees regarding the solution.} We consider precisely such an
	approach in  the context of economic dispatch problems, discussed in
	Section~\ref{sec:Numerics}. Table ~\ref{table:seq problem} shows the importance of terminating the
{learning problem} after a sufficiently large number of iterations via a
sequential approach. In particular, for smaller problems with $5$
generators, $15,000$ learning steps suffice in getting reasonably
accurate estimates of $x^*$ while the same number of iterations prove
insufficient for getting accurate solutions for networks with $20$
generators.  {In contrast}, our focus lies in developing techniques that
can provide {\bf asymptotically} accurate solutions equipped with
{global} {\bf non-asymptotic} error bounds.
\begin{table}[htbp]
\scriptsize
\begin{center}
\begin{tabular}{|c|c|c|c|c|c|}
  \hline
  \multirow{2}{*}{  Learning steps}&  \multirow{2}{*}{  Computational
	  steps} &  \multicolumn{2}{c|} {number of generators $=$ 5}&
	  \multicolumn{2}{c|}{ number of generator$=$ 40}\\\cline{3-6}
   &&$\|\theta_k-\theta^*\|$ &$\frac{\|f(x_k,\theta^*)-f^*\|}{1+f^*}$&$\|\theta_k-\theta^*\|$ &$\frac{\|f(x_k,\theta^*)-f^*\|}{1+f^*}$\\
  \hline  \hline
  1000&15000 &$6.43$e$0$& $6.49$e-$2$&$2.10$e$1$& $2.33$e-$1$\\\hline
	5000& 15000 &$3.34$e$0$ &$4.25$e-$2$  &$1.95$e$1$ &$9.05$e-$2$ \\\hline
	 10000&15000 &$1.48$e-$1$& $8.80$e-$3 $ &$1.77$e$1$& $8.08$e-$2 $ \\\hline
	15000 &15000&$1.63$e-$2$&$4.00$e-$4$&$1.06$e$1$&$6.60$e-$2$ \\\hline
\end{tabular}
\end{center}
\caption{Sequential approach: Effect of problem size on {accuracy}}
\label{table:seq problem}
\vspace{-0.2in}
\end{table}
\paragraph{A variational approach:} Given that a sequential approach
may not always be satisfactory, a partial resolution lies in considering a
variational approach where the overall problem is cast as a static
variational inequality problem~\cite{facchinei02finite}. If $X \subseteq \Real^n$ and $F: \Real^n
\to \Real^n$, then it may be recalled that VI$(X,F)$ requires an $x \in
X$ such that $ (y-x)^T F(x) \, \geq \, 0$ for all $\ y \ \in \ X. $
Under convexity assumptions, it can be shown with relative ease that $(x^*,\theta^*) \triangleq z^*$
is a solution of VI$(Z,H)$ if $Z \triangleq X \times \Theta$ and $H(z) =
(\nabla_{x} f(x,\theta);\nabla_{\theta} g(\theta)).$ But the solution of VI$(Z,H)$ via projection-based
techniques~\cite{Pang03I} remains a challenge since $H(z)$ is
generally not a monotone
map over the set $Z$ even if $f$ and $g$ are convex $C^1$ functions in $x$ and $\theta$,
	respectively;  it may be recalled that a map $H$ is monotone over $Z$ if
for every $z_1,z_2 \in Z$, $(H(z_1)-H(z_2))^T(z_1-z_2) \geq 0$. But
there are no available first-order schemes for computing solutions to
non-monotone variational inequality problems, severely limiting the
utility of such an approach. Yet, despite the inherently challenging
nature of the joint variational problem, our goal remains in deriving
non-asymptotic rates of convergence for gradient methods for such
problems by leveraging the structure of the problem and ascertaining the
impact that learning has on the rates.
\paragraph{A robust optimization approach:} Robust optimization, a
subfield of optimization, considers obtaining solutions that are robust
to parametric
uncertainty~\cite{Bertsimas:2011:TAR:2079012.2079015,Calafiore05uncertainconvex,BEN:09}.
In such problems, rather than a vector $\theta^*$, a part of the problem
input is the uncertainty set, say ${\cal U}_{\theta}$. In such a case, the relevant
robust optimizaton problem attempts to obtain an $x$ that minimizes the
worst-case value that $f(x,\theta)$ takes over ${\cal U}_\theta$:
\begin{align*}
\min_{x\in X}  \, \max_{\theta \in {\cal U}_\theta}  f(x,\theta).
\end{align*}
In contrast, our framework is fundamentally different in that the vector
$\theta^*$ is a {\bf deterministic} and unknown vector that can be
learnt.  To provide a clearer comparison, the learning scheme
	solves
the following problem
\begin{align*}
\min_{x\in X} \,f(x;\theta^*)\quad \mbox{where} \quad \theta^*=\argmin_{\theta \in \Theta} g(\theta).
\end{align*}

\paragraph{Learning while doing schemes:} 
Finally, we note the surge of interest in algorithms which
incorporate learning directly into the optimization phase. Early
instances of such problems were seen in the form of the multi-armed
bandit problem~\cite{katehakis87multiarmed,gittins89} in which a
decision-maker simultanelusly acquires new knowledge and leverages
existing knowledge in optimizing decisions.   In contrast
with the current context, the learning problem is no longer
static and available a priori; instead, it evolves in time as a
consequence of aggregating observations. In response to such challenges,
		   Agarwal et al. have developed techniques in the context on {\em online linear
programming}~\cite{agarwal14} as well as stylized counterparts in
the context of revenue management~\cite{agarwal11,wang14}. A rather
different tack is taken in the work by Jiang et
al.~\cite{jiang13distributed}, where the problem ${\cal L}$ is replaced
by a sequence of learning problems ${\cal L}_1, \hdots, {\cal L}_k$,
such that the index of $k$ represents the number of data points used
within the construction of the associated estimation (regression)
problem.  The solution of the $k$th learning problem is denoted by
$\theta_k$  and under suitable assumptions, $\{\theta_k\} \to \theta^*$,
where $\theta^*$ is a solution to the limiting problem.  If $\theta_k$
	is used within the scheme for computing $x$, then
	probabilistic convergence statements are provided for 
	$\{x_k\}$ in the context of distributed projection-based schemes for
	stochastic Nash games, leading to monotone variational inequality
			problems. We note that offline schemes provide a benchmark
			in terms of ascertaining the cost of obtaining observations
over time, rather than a priori, allowing us to derive metrics to relate
online schemes with their offline counterparts (such as through
		competitive ratios for instance).

\subsection{Contributions and outline}
In this paper, we investigate the global convergence and rate analysis
of joint first-order gradient methods under a variety of convexity,
   Lipschitzian, and boundedness requirements.  Suppose  $\gamma_{f,k}$
   and $\gamma_{g,k}$ denote steplength for optimization and learning at
   iteration $k$. If $\Pi_Y(y)$ denotes the Euclidean projection of a vector $y$ on the
set $Y$, then consider the following prototypical update:
\begin{algorithm}[{\bf Joint gradient scheme}]
\label{alg1}
{Given $x_0 \in X$ and $\theta_0 \in \Theta$ and sequences
	${\gamma_{f,k}, \gamma_{g,k}}$},
\begin{align}
x_{k+1} & := \Pi_X \left( x_k - \gamma_{f,k} \nabla_xf(x_k,\theta_k)\right), \quad
		& \forall  k \ \geq \ 0, \tag{Opt$(\theta_k)$} \\
\theta_{k+1} & := \Pi_{\Theta} \left( \theta_k - \gamma_{g,k}
	\nabla_{\theta} g(\theta_k)\right), \ \quad \quad
		& \forall  k \ \geq \ 0. \tag{Learn}
\end{align}
\end{algorithm}
}
{In our proposed scheme, we take a gradient step in the
	$x$-space using an estimate $\theta_k$ of $\theta^*$ and a
		simultaneous step in the $\theta-$space. Note that instead of using the exact gradient $\nabla_x
f(x_k,\theta^*)$ at the $k$th iterate, we employ $\nabla_x
	f(x_k,\theta_k)$ as the gradient estimate and $r_k=\nabla_x
	f(x_k,\theta_k)-\nabla_x f(x_k,\theta^*)$ represents the error in
	the gradient at iteration $k$.  Recent literature on inexact
	gradient schemes has investigated convergence properties and rate
	analysis for various schemes using inexact
	gradients with bounded
	error~\cite{Bertsekas:1999:GCG:588905.589301,d'Aspremont:2008:SOA:1653906.1653915,olivier:2013:MathProg,Necoara2014,NIPS2011_4452}.
	Our framework is distinct in that we develop a 	broader framework of
	gradient, extragradient, and regularized schemes for solving both
	optimization and variational inequality problems through the
	provision of modified algorithmic requirements (such as those on
			steplengths), asymptotics, and enhanced rate statements.
	The framework is developed under the caveat that the inexactness (in
	gradient estimates) decays to zero at a prescribed rate, a
	consequence of obtaining increasing accurate estimates of $\theta_k$
	when taking the gradient step in the $x-$space.  The main
	contributions of this work can be captured as follows:\\

\noindent {(i) \bf Convex optimization:} {In the first part of this
paper}, we develop asymptotics and rate statements for misspecified
convex optimization problems in smooth and nonsmooth settings and assume
that the learning problem is strongly convex, unless mentioned
otherwise: {(a) \em Smooth optimization problems:}  Our first set of results in the smooth regime demonstrate
that constant steplength schemes are convergent but lead to a
quantifiable decay in the linear convergence rate characteristic of
constant steplength gradient methods. Unfortunately, such techniques are
heavily reliant on the knowledge of certain problem parameters, in the
absence of which we show that diminishing steplength sequences are also
convergent. When the strong convexity assumptions are weakened, we note
that the presence of learning leads to modification of the convergence
rate in function values by an additive factor given by
$\|\theta_0-\theta^*\|{\cal
	O}\left(q_g^K + {1/K}\right)$ where $\theta_0$ represents our
	initial estimate of $\theta^*$, $q_g$ denotes the contractive
	constant in the learning problem. Finally, we demonstrate that when
	the learning problem loses strong convexity, under a suitably
	defined weak-sharpness requirement, global convergence can still be
	retained; {(b) \em Nonsmooth optimization problems:} When the
	optimization problem is nonsmooth, it can be shown that while the
	overall convergence rate of the proposed misspecified subgradient
	methods is still ${\cal O}({1/\sqrt K})$, a similar additive factor emerges
	of the form $\|\theta_0-\theta^*\|{\cal O}\left(q_g^K +
			{1/K}\right).$ A summary of the rate statements is provided
	in Table~\ref{tabrate}.\\
\begin{table}[htbp]
 \centering
\begin{tabular}{|c|c|c|c|c|c|c|c|c|c|c|c|}
  \hline
  & Computation & Computation \& Learning  \\
 \hline
  Strongly convex/diff. & Linear & Sublinear  \\
\hline
  convex/diff. & ${\cal O}(1/K)$ & ${\cal O}(1/K)+\|\theta_0-\theta^*\|{\cal O}(1/K + q_g^K)$  \\
\hline
  convex/nonsmooth. & ${\cal O}(1/\sqrt{K})$ & ${\cal
	  O}(1/\sqrt{K}) +\|\theta_0-\theta^*\| {\cal O}(1/K + q_g^K)$  \\
	  \hline
\end{tabular}
  \caption{Summary of rate statements}
  \label{tabrate}
\end{table}

\noindent {\bf (ii) Monotone variational inequality problems:} Variational
inequality problems represent a broad framework for capturing
optimization and equilibrium problems and assume particular relevance,
given that misspecification may arise in the constraints.  In the second
part of this paper, we consider two sets of schemes for resolving
misspecified variational inequality problems. Of these, the first avenue
is a constant steplenth misspecified extragradient scheme for monotone
variational inequality problems. However, this approach requires an
accurate estimate of suitable Lipschitz parameters. Consequently, we
present a  misspecified variant of the iterative Tikhonov regularization
framework to misspecified monotone regimes.\\

\noindent  {(iii) \bf  Numerics:} { We develop a set of test problems
	based on economic dispatch
	problems~\cite{kirschen2004a} with misspecified cost and demand. The
	numerics support the asymptotic statements and the validity of the
	bounds.

\section{Misspecified Convex Optimization}\label{sec:ConvOpt} 
In this section, we will consider two settings differentiated by the assumptions
on the function $f(x,\theta)$ in $(\cal C(\theta^*))$ and the function
$g(\theta)$ in $(\cal L)$. In Section~\ref{subsec:SmoothConvOpt}, we examine gradient-based methods where
$f(x,\theta)$ is differentiable in $x$ for every $\theta$ while in
Section~\ref{subsec:nonsmoothConvOpt}, we weaken the smoothness
requirement on $f(x,\theta)$. In each setting, we provide both constant
steplength schemes with associated complexity statements as well as
diminishing steplength schemes that are less reliant on problem
parameters. 
\subsection{Smooth convex optimization}\label{subsec:SmoothConvOpt}
{In this section, we consider regimes where both the optimization and
learning problems are differentiable and distinguish the cases based on
the convexity assumptions on the problem. Specifically, in subsection
~\ref{subsec:Str Opt Str learn}, we provide convergence statements and
rate analysis when both problems are strongly
convex.  Next, in subsection ~\ref{subsec: Conv opt Str learn}, we
weaken the strong convexity assumption on the computational problem to
mere convexity and provide rate statements in such settings.  Finally,
in subsection ~\ref{subsec: conv opt conv learn}, we relax the strongly
	convex assumption of the learning function and analyse the case when
	the solution set of the learning problem satisfies a weak sharpness
	assumption. We now list several key assumptions used during our
	analysis. We begin with a differentiability assumption on $f$ and
	$g$. }
\begin{assumption} \label{assump:differentiability}
	The {function} $f(x,\theta)$ is continuously differentiable in $x$ for all $\theta\in \Theta$ and function $g$ is
	continuously differentiable in $\theta$.
\end{assumption}

\noindent {Next, we impose a Lipschitzian assumption on $f$ in $x$, uniformly in
$\theta$.}
\begin{assumption} \label{assump:f_g_lips_contin_gradient}
           The gradient map $\nabla_xf(x;\theta)$ is Lipschitz
		   continuous in $x$  with constant $G_{f,x}$ uniformly over $\theta \in \Theta$ or
$$ \|\nabla_x f(x_1,\theta) - \nabla_x f(x_2,\theta) \|\leq G_{f,x}
\|x_1-x_2\|, \qquad \forall x_1, x_2 \in X, \quad \forall \theta \in \Theta. $$
Additionally,	the gradient map $\nabla_{\theta} g$ is Lipschitz continuous in $\theta$ with constant $G_g$.
\end{assumption}

\noindent {Finally, we impose a requirement on steplength sequences for the
computational and learning problems required in the diminishing
	steplength regime. }
\begin{assumption}\label{assump:step_length}
Let $\{\gamma_{f,k}\}$ and $\{\gamma_{g,k}\}$ be {diminishing}
nonnegative sequences chosen such that
	$\sum_{k=1}^\infty\gamma_{f,k}=\infty$,
	$\sum_{k=1}^\infty\gamma_{f,k}^2< \infty$,
	$\sum_{k=1}^\infty\gamma_{g,k}=\infty$, and
	$\sum\limits_{k=1}^\infty\gamma_{g,k}^2< \infty$.
\end{assumption}
\subsubsection{Strongly convex optimization and learning}\label{subsec:Str Opt Str learn}
In this subsection, convergence statements for the iterates
produced by Algorithm~\ref{alg1} are provided under the following strong
convexity assumption.
\begin{assumption}\label{assump:f_and_g_strconv}
The function $f$ is strongly convex in $x$ with constant $\eta_{f}$ for all $\theta \in \Theta$ and the function $g$ is strongly convex with constant $\eta_{g}$.
\end{assumption}

{We impose an additional Lipschitzian assumption on $\nabla_x
	f(x^*;\theta)$ in $\theta$.}
\begin{assumption}\label{assump:libs_contin_L_theta}
The gradient $\nabla_xf(x^*,\theta)$ is Lipschitz continuous in $\theta$ with constant $L_{\theta}$.
\end{assumption}

Before providing the main results, we introduce the following Lemma from \cite{Polyak87}:
\begin{lemma}
\label{polyak}
Let the following hold:
\begin{align*}
u_{k+1}\leq q_ku_k+\alpha_k,\quad 0\leq q_k<1, \quad \alpha_k\geq 0,
	\quad \sum_{k=1}^\infty(1- q_k)=\infty, \quad \lim_{k \to \infty} \frac{\alpha_k}{(1-q_k)} = 0.
\end{align*}
Then, $\lim_{k\to\infty} u_k\leq 0$. In particular, if $u_k\geq 0$, then $u_k \to 0$.
\end{lemma}

Our first result provides a convergence statement under a constant
steplength assumption.
\begin{proposition}[{\bf Constant step length scheme}]
\label{prop_conv_csl}
Let
Assumptions~\ref{assump:differentiability}~\ref{assump:f_g_lips_contin_gradient},~\ref{assump:f_and_g_strconv}
and ~\ref{assump:libs_contin_L_theta} hold and $\gamma_{f,k}$ and $\gamma_{g,k}$ are
fixed at $\gamma_{f}$ and $\gamma_{g}$, respectively so that
$0<\gamma_f< {2\slash G_{f,x}}$ and $0<\gamma_g<2/G_g$. Then, the sequence $\{x_k,\theta_k\}$ generated by Algorithm~\ref{alg1}
converges to $x^* \in X$ and $\theta^* \in \Theta$, respectively.
\end{proposition}
\begin{proof}
By nonexpansivity of the Euclidean projector and triangle inequality, $\|x_{k+1} - x^*\|$ can be bounded as follows:
\begin{align}
& \quad \, \|x_{k+1} - x^*\|\notag \\
		&=\| \Pi_X ( x_k - \gamma_{f} \nabla_xf(x_k,\theta_k))-\Pi_X ( x^* - \gamma_{f} \nabla_xf(x^*,\theta^*))\|\notag\\
                                &\leq \| (x_k-x^*) - \gamma_{f} (\nabla_xf(x_k,\theta_k) - \nabla_{x} f(x^*,\theta^*))\|\notag \\
			                   &\leq  \| (x_k-x^*) - \gamma_f
							   (\nabla_xf(x_k,\theta_k) - \nabla_x
								f(x^*, \theta_k)\|+ \gamma_f\| \nabla_xf(x^*,\theta_k) - \nabla_{x}f(x^*,\theta^*))\|.\label{con_expan1}
\end{align}
The first term in \eqref{con_expan1} can be further bounded by first writing the following expansion:
\begin{align}
  \| (x_k-x^*) - \gamma_f (\nabla_xf(x_k,\theta_k) - \nabla_x
		f(x^*, \theta_k) )\|^2\notag 
	&=  \| x_k-x^*\|^2 + \gamma_f^2 \|\nabla_xf(x_k,\theta_k) - \nabla_x f(x^*, \theta_k)\|^2
                     \\&  -2\gamma_f  (x_k-x^*)^T (\nabla_xf(x_k,\theta_k) - \nabla_x f(x^*, \theta_k)). \label{proof: ineq1}
\end{align}
Under the assumption of Lipschitz continuity of $\nabla_x f(x,\theta)$ in $x$, it follows that 
$$ \|\nabla_xf(x_k,\theta_k) - \nabla_x f(x^*, \theta_k)\|^2 \leq G_{f,x}(x_k-x^*)^T (\nabla_xf(x_k,\theta_k) - \nabla_x f(x^*, \theta_k)).$$
By combining the above inequality with~\eqref{proof: ineq1}, we obtain 
\begin{align}
 & \quad \| (x_k-x^*) - \gamma_f (\nabla_xf(x_k,\theta_k) - \nabla_x
		f(x^*, \theta_k) )\|^2\notag\\
                        & \le\| x_k-x^*\|^2 -\gamma_f(2-\gamma_fG_{f,x})(x_k-x^*)^T (\nabla_xf(x_k,\theta_k) - \nabla_x f(x^*, \theta_k)).\label{proof:ineq2}
\end{align}
In addition, under strong convexity of $f(x;\theta)$ in $x$, it follows that 
 \begin{align*}
 (x_k-x^*)^T (\nabla_xf(x_k,\theta_k) - \nabla_x f(x^*, \theta_k))\geq \eta_f \|x_k-x^*\|^2.
 \end{align*}
Thus, inequality~\eqref{proof:ineq2} becomes 
\begin{align}
\| (x_k-x^*) - \gamma_f (\nabla_xf(x_k,\theta_k) - \nabla_x
		f(x^*, \theta_k) )\|^2
	&\leq  \| x_k-x^*\|^2 
                       -\gamma_f\eta_f(2-\gamma_f G_{f,x})  \|x_k-x^*\|^2\notag\\
       &=\big(1-\gamma_f\eta_f(2-\gamma_f G_{f,x}) \big) \|x_k-x^*\|^2.\label{con_expan2}
\end{align}
Note that since $\gamma_f< ({2/G_{f,x}})$, then it follows that $\big(1-\gamma_f\eta_f(2-\gamma_f G_{f,x}) \big) <1.$ The second term in $\eqref{con_expan1}$ is bounded by leveraging the
 Lipschitz continuity of $\nabla_{x} f(x^*,\theta)$ in $\theta$:
\begin{align}
\label{con_tta_bnd}
&\| \gamma_f(\nabla_xf(x^*,\theta_k) - \nabla_{x} f(x^*,\theta^*))\|\leq
\gamma_f L_{\theta} \| \theta_k-\theta^*\|.
\end{align}
Now by combining \eqref{con_expan1}, \eqref{con_expan2}, and
\eqref{con_tta_bnd}, we obtain the following bound:
\begin{align}
\label{bnd_csl}
                  \|x_{k+1} - x^*\|&  \leq q_{x} \| x_k-x^*\| +
				  q_{\theta} \| \theta_k-\theta^*\|, 
\end{align}
where $q_{x}\triangleq   \sqrt{\big(1-\gamma_f\eta_f(2-\gamma_f G_{f,x}) \big)}$
and $q_{\theta}\triangleq \gamma_f L_{\theta}$. To show that $\|x_k -
x^*\|\to 0$ as $k\to\infty$, we may employ Lemma~\ref{polyak}. This
requires showing the following:
\begin{align*}
{ (i)} \quad \sum_{k=1}^\infty(1-q_{x})=\infty; \qquad
{ (ii)} \quad \lim_{k\to\infty} \frac{q_{f,\theta} \|
\theta_k-\theta^*\|}{1-q_{x}}=0.
\end{align*}
Since $q_{x}<1$, $(i)$ is
satisfied. In addition, by the Lipschitz continuity of $\nabla_{\theta} g$
and choosing $\gamma_g$ such that $0<\gamma_g<2/G_g$, $\|\theta_k-\theta^*\|\to 0$ as
$k\to\infty$. Consequently,  condition $(ii)$ is met as well,
	completing
the proof.
\end{proof}

In many instances, while we may be able to claim strong convexity or
Lipschitz continuity, the precise bounds may be unavailable. However,
an incorrect choice of a steplength may lead to divergence, motivating
the need for an alternate approach. To this end, we employ a diminishing
steplength sequence that does not necessitate the knowledge of either
the convexity constant or the Lipschitz constant.  We outline the proof
of convergence in the next Proposition.\\

\begin{proposition}[{\bf Diminishing steplength schemes}]
\label{Prop:dimin_step}
Let
Assumptions~\ref{assump:differentiability},~\ref{assump:f_g_lips_contin_gradient},~\ref{assump:f_and_g_strconv},
	and ~\ref{assump:libs_contin_L_theta} hold. Additionally, let
		$\gamma_{f,k}$ be defined based on Assumption
			~\ref{assump:step_length} and {$\gamma_{g,k}$ be fixed at $\gamma_g$ so that $0<\gamma_g < 2/G_g$.}Then, the sequence
	$\{x_k,\theta_k\}$ generated by Algorithm \ref{alg1}
converges to $x^* \in X$ and $\theta^* \in \Theta$, respectively.
\end{proposition}
\begin{proof}
We use a similar line of argument as in Proposition~\ref{prop_conv_csl} 
to obtain the following bound:
\begin{align}
   \|x_{k+1} - x^*\| \leq q_{x,k} \| x_k-x^*\| + q_{\theta,k} \|
   \theta_k-\theta^*\|,
\end{align}
where for sufficiently large $k$, we have that $q_{x,k}\triangleq {(1-  \gamma_{f,k}\eta_{f}(2-\gamma_{f,k}
			G_{f,x}))^{1/2}} < 1$
and $q_{\theta,k}\triangleq \gamma_{f,k} L_{\theta}$. By Assumption~\ref{assump:step_length}, we have that
$\sum_{k=1}^\infty(1-q_{x,k})=\infty$. Furthermore,  we have the
following simplification of condition (ii) of Lemma~\ref{polyak}:
\begin{align*}
\lim_{k\to\infty}\frac{\gamma_{f,k} L_{\theta}\| \theta_k-\theta^*\|}
{(1-q_{x,k})} & =
\lim_{k\to\infty} \frac {(1+q_{x,k}) L_{\theta}\| \theta_k-\theta^*\|}{{\eta_{f}(2-\gamma_{f,k}G_{f,x})}}=0
\end{align*}
since  $q_{x,k}\to 0$ and $\gamma_{f,k}\to 0$ and
$\|\theta_k-\theta^*\|\to0$ as $k\to\infty$. Therefore, the conditions of Lemma~\ref{polyak}
are satisfied and $\|x_k-x^*\|\to0$  as $k\to\infty$.
\end{proof}

It is well known that under strong convexity assumption, the
{iterates generated from the}  projected
gradient method converge at a geometric rate~\cite{bertsekas03convex}.
However, when learning is incorporated, it is expected that this rate
drops. Next, we analyze the impact introduced by learning.

\begin{proposition}[{\bf Rate analysis in strongly convex regimes}]
\label{Prop: con_rate_strconv}
Let
Assumptions~\ref{assump:differentiability},~\ref{assump:f_g_lips_contin_gradient},~\ref{assump:f_and_g_strconv}
and ~\ref{assump:libs_contin_L_theta} hold.  In addition, assume that
$\gamma_{f}$ and $\gamma_{g}$ are chosen such that $0<\gamma_{f} <
2/G_{f,x}$ and $0<\gamma_g {<}
{2/G_g}$. Let $\{x_k,\theta_k\}$ be the sequence generated by
Algorithm 1. Then for every $k \geq 0$,
  we have the following:
\begin{align*}
 \|x_{k+1}-x^*\| & \leq q_{x}^{k+1} \|x_0-x^*\| + (k+1)q_{\theta} q^k
	\|\theta_0-\theta^*\|,
\end{align*}
where $q_{x}\triangleq (1-\gamma_{f} \eta_{f}(2-\gamma_f G_{f,x}))^{1/2}$, $q_{\theta}\triangleq
\gamma_{f} L_{\theta}$, $q_{g}\triangleq (1-\gamma_{g}
		\eta_{g}(2-\gamma_g G_{g}))^{1/2}$,
and $q\triangleq \max(q_x,g_g)$.
\end{proposition}
\begin{proof}
Under the assumption of strong convexity of $g$, the learning algorithm
has a globally geometric rate of convergence when employing constant
stepsize $\gamma_g$ where $0<\gamma_g < 2/G_g$; specifically,
\begin{align}
\label{geom_cr}
 \|\theta_{k+1}-\theta^*\|\leq q_{g}^{k+1} \|\theta_0-\theta^*\|, \quad
 \forall k \ \geq  \ 0.
\end{align}
{where $q_{g}\triangleq (1-\gamma_{g}
		\eta_{g}(2-\gamma_g G_{g}))^{1/2}<1$ since $\gamma_g<2/G_{g}$
}. To obtain the convergence rate for the joint scheme, we expand $\eqref{bnd_csl}$ to obtain the following bound:
\begin{align}
\|x_{k+1}-x^*\| & \leq q_{x}^{k+1}\|x_0-x^*\|
+q_{\theta}\sum_{i=0}^k q_{x}^i\|\theta_{k-i}-\theta^*\| \qquad
\forall \  k \geq 0 .
\label{bnd_expn}
\end{align}
We may further expand $\eqref{bnd_expn}$ using $\eqref{geom_cr}$ to simplify the bound as below:
\begin{align*}
 \|x_{k+1}-x^*\| & \leq q_{x}^{k+1}\|x_0-x^*\|+q_{\theta}
 \sum_{i=0}^k q_{x}^i q_g^{k-i}\|\theta_0-\theta^*\|\leq
 q_{x}^{k+1} \|x_0-x^*\| + \underbrace{(k+1)q_{\theta} q^k
	\|\theta_0-\theta^*\|,}_{\tiny \mbox{Degradation from learning}}
\end{align*}
where $q\triangleq\max(q_{x},q_g)$. Note that condition {$\gamma_f <
	2/G_{f,x}$} guarantees that {$q_{x} <1$}, implying that $q =
	\max(q_x,q_g)$ is less than 1.
\end{proof}

{\bf Remark:} Notably, the presence of learning leads to a degradation
in the convergence rate from the standard linear rate to
a sub-linear rate. Furthermore, it is easily seen that when we have
access to the true $\theta^*$,  the original rate may be recovered.

\subsubsection{Convex optimization with strongly convex learning}\label{subsec: Conv opt Str learn}
In this subsection, we weaken the rather stringent assumptions of
strong convexity of $f(x,\theta)$ in $x$ for every $\theta \in \Theta$.

{\sc {Assumption}}~\ref{assump:f_and_g_strconv}b.
{\em The function $f$ is convex in $x$  for all
	$\theta \in \Theta$ and the function $g$ is strongly convex with
		constant $\eta_{g}$.}

In addition, we make the following assumptions:
\begin{assumption}
\label{assump:compact and  lips cont}
	\be
	\item[(a)] The sets $X$ and $\Theta$ are compact and $\sup_{x\in X}\|x\|\leq C$, where $C$ is a constant.
    \item[(b)] The gradient map $\nabla_x f(x;\theta)$ is uniformly Lipschitz continuous in $\theta$ with constant $G_{f,\theta}$:
       $$\|\nabla_x f(x,\theta_1)-\nabla_x f(x,\theta_2)\|\leq
	   G_{f,\theta}\|\theta_1-\theta_2\|, \quad \forall
	   \theta_1,\theta_2\in \Theta, x\in X.$$
	\ee
\end{assumption}
\begin{assumption}
\label{assump:lips_cont}
There exists a constant $L_{f,\theta}$ such that
$|f(x,\theta_1)-f(x,\theta_2)|\leq L_{f,\theta}\|\theta_1-\theta_2\|, \quad \forall \theta_1,\theta_2 \in \Theta, x\in X.$
\end{assumption}

Before presenting the main result, we introduce the following Lemma from
\cite{MR0343355}.
\begin{lemma}
\label{polyak2}
Let  $\beta_k, \upsilon_k, \alpha_k\geq 0$ for all $k$. Furthermore,
	 suppose the following holds for all $k$:
\begin{align*}
u_{k+1}\leq (1+\beta_k)u_k-\upsilon_k+\alpha_k.
\end{align*}
Suppose $\sum_{k} \alpha_k < \infty$ and $\sum_{k} \beta_k < \infty$.
Then  $\lim_{k\to\infty} u_k = \bar u \geq 0$ and $\sum_k \upsilon_k <
\infty$.
\end{lemma}

In the following proposition, we prove the convergence of the iterates
generated by Algorithm ~\ref{alg1} under the convexity requirements on
$f(x;\theta)$. We also provide the rate statement.\\

\begin{proposition}[{\bf Constant steplength scheme with averaging}]
\label{theo_conv_convex}
Let Assumptions ~\ref{assump:differentiability},~\ref{assump:f_g_lips_contin_gradient},~\ref{assump:f_and_g_strconv}b,~\ref{assump:compact and
		lips cont} and ~\ref{assump:lips_cont} hold and stepsizes
		$\gamma_{f,k}$ and $\gamma_{g,k}$ be fixed at
		constants $\gamma_{f}$ and $\gamma_{g}$ so that
$0<\gamma_g<2/G_g$ and $0<\gamma_f \leq 1/ G_{f,x}$. Let the
sequence $\{x_k,\theta_k\}$ be generated by Algorithm ~\ref{alg1} and
suppose $\xbar_k$ is defined as $$\xbar_{k} \triangleq
\frac{\displaystyle \sum_{i=0}^{k-1} x_{i+1}}{k}.$$
Then the following hold:
\begin{enumerate}
\item[(i)] In addition, if $a_x =
\frac{\|x_0-x^*\|^2}{2\gamma_f}$ and
$b_\theta \triangleq \frac{CG_{f,\theta}}{1-q_g},$  then the following holds:
 \begin{align*}
\left| f\left( \xbar_{K},\theta_K \right) -f(x^*,\theta^*)\right| \leq
\frac{a_x}{K} + \|\theta_0-\theta^*\|\left(\frac{b_{\theta}}{K} + L_{f,\theta} q_g^K \right).
\end{align*}
\item[(ii)] ${\displaystyle \lim_{k\to \infty}}  f\left(\xbar_{k},\theta_k \right) =f(x^*,\theta^*).$
\end{enumerate}
\end{proposition}
\begin{proof}
(i) \ Recall the following by the mean-value theorem, the Cauchy-Schwartz
inequality, and {Lipschitz continuity of the gradient map $\nabla_x
f(x,\theta^*)$} in $x$:
\begin{align*}
		    f(y,\theta^*)& = f(x,\theta^*) + \nabla_x
			f(x,\theta^*)^T(y-x)+ \int_0^1 (\nabla_x
					f(x+t(y-x),\theta^*) - \nabla_x f(x,\theta^*)) (y-x) dt  \\
					& \leq f(x,\theta^*) +  \nabla_x
					f(x,\theta^*)^T(y-x) + \int_0^1 \|(\nabla_x
					f(x+t(y-x),\theta^*) - \nabla_x
					f(x,\theta^*))\|\|(y-x)\| dt  \\
			& \leq f(x,\theta^*) +  \nabla_x f(x,\theta^*)^T(y-x) + \int_0^1 G_{f,x}t\|y-x\|\|(y-x)\| dt  \\
			& = f(x,\theta^*) +  \nabla_x f(x,\theta^*)^T(y-x) + {1 \over 2}G_{f,x}\|y-x\|^2.
\end{align*}
If we set $y=x_{i+1}$ and $x=x_i$ and since $G_{f,x}\leq {1\over \gamma_f}$, we have the following:
\begin{align} \notag
 f(x_{i+1},\theta^*) & \leq f(x_i,\theta^*) +(\nabla_x f(x_i,\theta_i)-r_i)^T(x_{i+1}-x_i) + {G_{f,x}\over 2}\|x_{i+1}-x_i\|^2\\
 &= f(x_i,\theta^*) + \nabla_x f(x_i,\theta_i)^T(x_{i+1}-x_i)+ {1\over
	 2\gamma_f}\|x_{i+1}-x_i\|^2 -r_i^T(x_{i+1}-x_i),\label{ineq01}
\end{align}
where $r_i\triangleq\nabla_x f(x_i,\theta_i)-\nabla_x f(x_i,\theta^*)$.
Under the convexity of $f(x;\theta^*)$ in $x$,
\begin{align}
f(x_i,\theta^*) &\leq
f(x^*,\theta^*)+\nabla_x f(x_i,\theta^*)^T(x_i-x^*)
\notag\\&= f(x^*,\theta^*)+\nabla_x
f(x_i,\theta_i)^T(x_i-x^*)-r^T_i(x_i-x^*).\label{ineq02}
\end{align}
By summing up \eqref{ineq01} and \eqref{ineq02}, we obtain
\begin{align}
 f(x_{i+1},\theta^*)
\leq  f(x^*,\theta^*) + \nabla_x f(x_i,\theta_i)^T(x_{i+1}-x^*)+ {1\over
	 2\gamma_f}\|x_{i+1}-x_i\|^2 -r_i^T(x_{i+1}-x^*).\label{ineq1}
\end{align}
Next, we bound the term $\nabla_x f(x_i,\theta_i)^T(x_{i+1}-x^*).$
From the property of the projection on a convex set, denoted by
$\Pi_X(x)$, we have that
\begin{align*}
(x-\Pi_X(x))^T(y-\Pi_X(x))\leq 0,  \quad \forall x\in \R^n, y\in X.
\end{align*}
If we set $x=x_i-\gamma_f \nabla f(x_i,\theta_i)$ and $y=x^*$ {in the
above inequality and by noting that $x_{i+1} = \Pi_X(x)$}, we obtain that
$(x_i-\gamma_f\nabla_x f(x_i,\theta_i)-x_{i+1})^T(x^*-x_{i+1})\leq 0.$
After rearrangement of the terms, the above inequality is equivalent to
\begin{align*}
\nabla_xf(x_i,\theta_i)^T(x_{i+1}-x^*)\leq {1\over \gamma_f}(x_{i+1}-x_i)^T(x^*-x_{i+1}).
\end{align*}
By using this bound in ~\eqref{ineq1}, we get that
\begin{align*}
f(x_{i+1},\theta^*)
\leq  f(x^*,\theta^*) +  {1\over \gamma_f}(x_{i+1}-x_i)^T(x^*-x_{i+1})+ {1\over
	 2\gamma_f}\|x_{i+1}-x_i\|^2 -r_i^T(x_{i+1}-x^*).
\end{align*}
Since $\|x_{i+1}-x_i\|^2+2(x_{i+1}-x_i)^T(x^*-x_{i+1})=\|x_i-x^*\|^2-\|x_{i+1}-x^*\|^2$, the above inequality can be written as
\begin{align*}
f(x_{i+1},\theta^*)
\leq  f(x^*,\theta^*) +  {1\over 2\gamma_f}\|x_{i}-x^*\|^2- {1\over
	 2\gamma_f}\|x_{i+1}-x^*\|^2 -r_i^T(x_{i+1}-x^*).
\end{align*}
Moving $f(x^*,\theta^*)$ to the other side and summing from $i=0$ to
$K-1$, we get the following:
\begin{align*}
\sum_{i=0}^{K-1}\left(f(x_{i+1},x^*)-f(x^*,\theta^*)\right)& \leq
-{1\over 2\gamma_f}\|x_{K}-x^*\|^2+{1\over
	2\gamma_f}\|x_0-x^*\|^2+\sum_{i=0}^{K-1}\|r_i\|\|x_{i+1}-x^*\| \\
 &  \leq
{1\over
	2\gamma_f}\|x_0-x^*\|^2+\sum_{i=0}^{K-1}\|r_i\|\|x_{i+1}-x^*\|,
\end{align*}
where the second inequality follows from the nonnegativity of  ${1\over
	2\gamma_f}\|x_{K}-x^*\|^2$. Dividing both sides by $K$,
\begin{align}
{1\over
	K}\sum_{i=0}^{K-1}\left(f(x_{i+1},x^*)-f(x^*,\theta^*)\right)\leq
{1\over 2\gamma_fK}\|x_0-x^*\|^2+{1\over K}\sum_{i=0}^{K-1}\|r_i\|\|x_{i+1}-x^*\|.\label{con_rate_cnv_bnd}
\end{align}
 By Assumption ~\ref{assump:compact and  lips cont}(a), $\|x_{i+1}-x^*\|\leq C$ for all $i\geq 0$. In addition, by Assumption ~\ref{assump:compact and  lips cont}(b), we have that
 $\|r_i\|=\|\nabla_x f(x,\theta_i)-\nabla_x f(x,\theta^*)\|\leq
 G_{f,\theta}\|\theta_i-\theta^*\|$. Since the function $g$ is strongly
 convex and $\gamma_g\leq {2 \over G_g}$, there exists a $q_g\in(0,1)$
 such that $\|\theta_i-\theta^*\|\leq q_g^i\|\theta_0-\theta^*\|$.
 Therefore, from ~\eqref{con_rate_cnv_bnd}, we obtain the following:
\begin{align*}
{1\over K}\sum_{i=0}^{K-1} (f(x_{i+1},\theta^*)-f(x^*,\theta^*)) & \leq {1\over
	2\gamma_fK}{\|x_0-x^*\|^2} +C G_{f,\theta}\|\theta_0-\theta^*\|\frac{(1-q_g^K)}{K(1-q_g)}.
\end{align*}
By leveraging the convexity of $f(\bullet;\theta^*)$ in $(\bullet)$, we have that
	\begin{align}
	\quad f\left( \xbar_{K},\theta^* \right) -
f(x^*,\theta^*)\leq {1\over
	2\gamma_fK}{\|x_0-x^*\|^2}
	+CG_{f,\theta}\|\theta_0-\theta^*\|\frac{(1-q_g^K)}{K(1-q_g)}.\label{prop:rate ineq}
			\end{align}
But, we may derive a bound on $\left|f\left( \xbar_{k},\theta_K \right)
	- f(x^*,\theta^*)\right|$ as follows:
\begin{align}
\label{prop:absolute ineq}
\left|f\left( \xbar_{K},\theta_K \right) -
f(x^*,\theta^*)\right|  \leq \left| f\left( \xbar_{K},\theta_K \right) -
f\left( \xbar_{K},\theta^* \right)\right|  + \left|f\left(\xbar_{K},\theta^* \right) -
f(x^*,\theta^*)\right|.
\end{align}
{We leverage the Lipschitz continuity of $f(x,\theta)$ in $\theta$
	{uniformly in $x$} with
	constant $L_{f, \theta}$ together with
		~\eqref{prop:rate ineq} and ~\eqref{prop:absolute ineq} to
		complete the proof of (i)}:
\begin{align}
\left| f\left( \xbar_{K},\theta_K \right) -f(x^*,\theta^*)\right| \leq
\frac{a_x}{K} + \underbrace{\|\theta_0-\theta^*\|\left( L_{f,\theta}q_g^K
	+\frac{b_{\theta}}{K}\right)}_{\tiny \mbox{Impact of
		learning}}\label{rate_average}
\end{align}
where $a_x =
\frac{\|x_0-x^*\|^2}{2\gamma_f}$ and
$b_\theta \triangleq \frac{CG_{f,\theta}}{1-q_g},$ . \\
\noindent (ii) Global convergence
follows by taking limits \eqref{rate_average} and by recalling that
$q_g < 1$ to claim that
\begin{align*}
\lim_{k\to \infty}  \left|f\left( \xbar_{k},\theta_k \right) -
f(x^*,\theta^*)\right|= 0.
\end{align*}
\end{proof}

\noindent {\bf Remark:} {Unlike in the case of strongly convex
optimization, there is {\bf no} degradation in the standard rate of
convergence in function values which is ${\cal O}(1/K)$. In particular, the
contribution from learning adds a factor to this rate that is scaled
by $\|\theta_0-\theta^*\|$, the distance of $\theta_0$ from $\theta^*$.  Notably,
this factor has two parts, the first of which is a faster geometric rate
given by $L_{f,\theta}q_g^K$ and the a second part given by
$b_{\theta}/K$. In short, the overall rate
changes by a constant factor. Furthermore, if $\theta_0 = \theta^*$, we
recover the standard rate for convex optimization. However, this scheme
does require knowledge of relevant Lipschitz
constants and we now present diminishing
steplength schemes that do require Lipschitzian properties but do not
require knowing the precise constants.\\

\begin{proposition}[{\bf Diminishing steplength scheme}]
\label{Prop: Conv Diminish Step}
Let Assumptions ~\ref{assump:differentiability},~\ref{assump:f_g_lips_contin_gradient}, \ref{assump:f_and_g_strconv}b, and ~\ref{assump:compact and
		lips cont}  hold. Additionally, let
		$\gamma_{f,k}$ be defined based on Assumption
			~\ref{assump:step_length} and $\gamma_g < 2/G_g$. Let  {the sequence} $\{x_k,\theta_k\}$  be
generated by Algorithm ~\ref{alg1}.  Then, $\{x_k\}$ converges to a point in $X^*$ and $\{\theta_k\}$ converges to {$\theta^*\in \Theta$}.
\end{proposition}
\begin{proof}
By the nonexpansivity property of the Euclidean projection operator,
   for all $k>0$, $\|x_{k+1}-x^*\|^2$ can be bounded as follows:
\begin{align}
\|x_{k+1}-x^*\|^2& =\| \Pi_X ( x_k - \gamma_{f,k}
		\nabla_xf(x_k,\theta_k))-{\Pi_X(x^*)}\|^2\notag\\
&\leq \| (x_k-x^*) - \gamma_{f,k}\nabla_xf(x_k,\theta_k)\|^2\notag \\
                               &=\|x_k-x^*\|^2-2\gamma_{f,k}\nabla_xf(x_k,\theta_k)^T(x_k-x^*) +\gamma_{f,k}^2\|\nabla_xf(x_k,\theta_k)\|^2\notag \\
                              & =
\|x_k-x^*\|^2-2\gamma_{f,k}\nabla_xf(x_k,\theta^*)^T(x_k-x^*)
-2\gamma_{f,k}r_{k}^T(x_k-x^*)\notag \\
	& +\gamma_{f,k}^2\|\nabla_xf(x_k,\theta_k)\|^2,  \label{bd1}
\end{align}
where $r_k\triangleq\nabla_xf(x_k,\theta_k)-\nabla_xf(x_k,\theta^*)$. By leveraging convexity and the gradient inequality, we have that
$f(x^*,\theta^*) \geq
f(x_k,\theta^*)+\nabla_xf(x_k,\theta^*)^T(x^*-x_k),$
implying that
\begin{align}
   -\nabla_xf(x_k,\theta^*)^T(x_k-x^*)\leq
   -(f(x_k,\theta^*)-f(x^*,\theta^*)). \label{bd-rk}
\end{align}
By substituting \eqref{bd-rk} in \eqref{bd1} and by noting that
$2\gamma_{f,k}r_{k}^T(x_k-x^*) \leq \|r_k\|^2 + \gamma^2_{f,k} \|x_k -
x^*\|^2$,  we have the following bound:
\begin{align}
\|x_{k+1}-x^*\|^2&\leq
\|x_k-x^*\|^2-2\gamma_{f,k}(f(x_k,\theta^*)-f(x^*,\theta^*))-2\gamma_{f,k}r_{k}^T(x_k-x^*)
+\gamma_{f,k}^2\|\nabla_xf(x_k,\theta_k)\|^2 \notag\\
&\leq\|x_k-x^*\|^2-2\gamma_{f,k}(f(x_k,\theta^*)-f(x^*,\theta^*))+\|r_k\|^2\notag\\&+
\gamma^2_{f,k}\|x_k-x^*\|^2
+\gamma_{f,k}^2\|\nabla_xf(x_k,\theta_k)\|^2. \label{con_conv_bnd}
\end{align}
By Assumption ~\ref{assump:compact and  lips cont}, we have that $\|r_k\|^2\leq \|\nabla_xf(x_k,\theta_k)-\nabla_xf(x_k,\theta^*)\|^2\leq G^2_{f,\theta}\|\theta_k-\theta^*\|^2$.
In addition, under strong convexity of $g$ and choosing $\gamma_g <
2/G_g$, we have that $\|\theta_k-\theta^*\|^2\leq
q^{2k}_{g}\|\theta_0-\theta^*\|^2$, where $q_g\in (0,1)$. {Consequently
	$\theta_k \to \theta^*$ as $k \to \infty$. Furthermore}, $\eqref{con_conv_bnd}$ can be further simplified as below:
\begin{align}
\|x_{k+1}-x^*\|^2&\leq
(1+
 \gamma^2_{f,k})\|x_k-x^*\|^2-2\gamma_{f,k}(f(x_k,\theta^*)-f(x^*,\theta^*))\notag\\&\underbrace{+G_{f,\theta}^2q_g^{2k}
\|\theta_0-\theta^*\|^2
	+\gamma_{f,k}^2\|\nabla_xf(x_k,\theta_k)\|^2}_{\triangleq \alpha_k}.\label{eq:summable_alpha}
\end{align}
The requirements of Lemma~\ref{polyak2} hold since
$(f(x_k,\theta^*)-f(x^*,\theta^*)) \geq 0$ since $x^* \in \argmin_{x \in
X} f(x;\theta^*).$ Consequently, by leveraging Lemma~\ref{polyak2}, we observe that $\sum_{k=1}^{\infty}
\alpha_k < \infty$ since $$\sum_{k}{ G_{f,\theta}^2}q_g^{2k} \|\theta_0-\theta^*\|^2 \leq
\frac{{ G_{f,\theta}^2}\|\theta_0-\theta^*\|^2}{1-q_g^{2}}$$ and  $\sum_{k}
\gamma_{f,k}^2\|\nabla_xf(x_k,\theta_k)\|{^2}   <    \infty$, since
$\sum_{k}
\gamma_{f,k}^2 < \infty$ and
$\|\nabla_xf(x_k,\theta_k)\|$ is bounded, { a consequence of the
	compactness of $X$ and $\Theta$ and the continuity of the gradient
		map}. We may therefore conclude that $\|x_{k}-x^*\|^2 \to \bar v \geq 0$
and $\sum_{k} \gamma_{f,k} (f(x_k,\theta^*)-f(x^*,\theta^*)) < \infty.$
It suffices to show that $\bar v \equiv 0$.\\

Since $\sum_{k} \gamma_{f,k} (f(x_k,\theta^*)-f(x^*,\theta^*)) <
\infty$ and $\sum_{k} \gamma_{f,k} = \infty$, it follows that  $\liminf_{k\to\infty}f(x_k,\theta^*)
	=f(x^*,\theta^*)$.  Since the set $X$ is closed, all accumulation
	points of $\{x_k\}$ lie in $X$. Furthermore, since
	$f(x_k,\theta^*)\to f(x^*,\theta^*)$ along a subsequence,  it follows that $\{x_k\}$ has a subsequence
	converging to some point in $X^*$. Moreover, since $\|x_{k}-x^*\|$
	is convergent, then the entire sequence  $\{x_k\}$
	converges to a point in $X^*$.
\end{proof}
\subsubsection{Convex optimization with convex learning} \label{subsec: conv opt conv learn}

{A key restriction in the prior subsection is the need for imposing a
	strong convexity assumption on the learning problem. The need for
		this assumption arises from noting that we require utilizing a
		rate estimate in solution iterates in the learning space, rather
		than merely function iterates. In this
		subsection, we consider a convex learning problem but impose a
		weak sharpness requirement~\cite{Polyak87} which is defined next. Note that an alternative approach is
		pursued in the next section in a more general variational
		regime.}

\begin{definition}[{\bf Weak sharpness}]
The solution set $\Theta^*$ is said to be weak sharp if there
{exists} a positive number $\alpha$ such that $g(\theta)-g(\theta^*)
	\geq \alpha \mbox{dist}(\theta,\Theta^*), \quad  \forall \theta^* \in \Theta^*,$ where $\mbox{dist}(\theta,\Theta^*):=\min_{{\theta^*}\in \Theta^*} \|\theta-\theta^*\|$ and $\alpha$ is called modulus of sharpness.
\end{definition}

Under a weak sharpness requirement on the solution set, the solution to
the learning problem can be obtained in a finite number of iterations.
The proof of this Lemma may be found in~\cite{Polyak87}.

\begin{lemma}[{\em Finite convergence under constant steplength}]
\label{Lemma: Weak_sharp_finite}
Consider a convex differentiable learning  problem ${\cal L}$ in which the solution set
$\Theta^*$ is nonempty and satisfies a weak sharpness property.
Furthermore, $\nabla_{\theta} g$ is assumed to be Lipschitz continuous
with a constant $G_{g}$. Then, the sequence $\{\theta_k\}$ generated by
a projected gradient scheme with stepsize $\gamma_{g}< {2\over G_g}$
converges to $\theta^*$ in a finite number of iterations, where
$\theta^*\in \Theta^*$.
\end{lemma}

{
{We now consider a constant steplength scheme where $\gamma_{f,k}$ and $\gamma_{g,k}$ are
	sufficiently small constants.}

\begin{proposition}[{\bf Constant steplength scheme}]
Let Assumptions ~\ref{assump:differentiability},~\ref{assump:f_g_lips_contin_gradient}, and ~\ref{assump:compact and
		lips cont}  hold. In addition, suppose that  $\Theta^*$ {satisfies a weak sharpness requirement} and
the stepsize sequences
$\{\gamma_{f,k}\}$
and $\{\gamma_{g,k}\}$ are fixed at some positive constants $\gamma_f$ and $\gamma_g$, respectively, where
$0<\gamma_f< 2/G_{f,x}$ and $0<\gamma_{g}< 2/G_g$. Let $\{x_k,\theta_k\}$ be the sequence generated by Algorithm ~\ref{alg1}. Then, $\{x_k\}$ converges to a point in $X^*$ and  $\{\theta_k\}$ converges to a point { in $\Theta^*$} as $k\to\infty$.
\end{proposition}
\begin{proof}
Based on Lemma ~\ref{Lemma: Weak_sharp_finite}, there exist a finite
$K>0$ such that for all $k>K$, we have that $\theta_k=\theta^*\in
\Theta^*$. Hence, for all $k>K$, Algorithm ~\ref{alg1}, becomes standard
projected gradient scheme without learning and thus under Lipschitzian
property of gradient of function $f$ and by choosing $0<\gamma_f<
2/G_{f,x}$, the sequence $\{x_k\}$ converges to $x^*\in X^*$. For the
proof of convergence of gradient projected scheme, the reader can refer
to~\cite{Polyak87} .
\end{proof}}

{Next, we consider a diminishing steplength sequence for the {optimization} and learning
problems and provide an intermediate result on the summability of the
sequence $\{\gamma_{g,k} \mbox{dist}(\theta_k, \Theta^*)\}.$}
\begin{lemma}\label{lemma:rate_ws}
Consider a convex differentiable learning problem ${\cal L}$ in which the solution set $\Theta^*$ is nonempty and satisfies a weak sharpness property.
In addition, suppose that $\Theta$ is bounded and the sequence $\gamma_{g,k}$ be defined based on Assumption
			~\ref{assump:step_length}. Then, for the sequence $\{\theta_k\}$ generated by Algorithm~\ref{alg1}, we have that
{$ \sum_{k=1}^\infty  \gamma_{g,k} \mbox{dist}(\theta_k,
		 \Theta^*)<\infty.$}
\end{lemma}
\begin{proof}
Under boundedness of gradient of function $g$ and by using diminishing step length
\begin{align*}
\|\theta_{k+1}-\theta^*\|^2\leq
\|\theta_k-\theta^*\|^2-2\gamma_{g,k}(g(\theta_k)-g(\theta^*))+\gamma_{g,k}^2
\|\nabla_{\theta} g(\theta)\|^2.
\end{align*}
Under the weak sharp property of $\Theta^*$, we have that $g(\theta_k)-
g(\theta^*) \geq \alpha \mbox{dist}(\theta_k,\Theta^*)$. By {substituting
	this expression} into the above inequality, we obtain
$$\|\theta_{k+1}-\theta^*\|^2 \leq \|\theta_{k}-\theta^*\|^2-2\alpha \gamma_{g,k}\mbox{dist}(\theta_k,\Theta^*)+\gamma_{g,k}^2 C^2,$$
where $C:=\sup_{\theta \in \Theta} \|\nabla g(\theta)\|$.
Since $\sum_{k=1}^\infty \gamma_{g,k}^2 C^2 < \infty$, then by using Lemma ~\ref{polyak2}, we conclude that
$\sum_{k=1}^\infty  \gamma_{g,k}\mbox{dist}(\theta_k,\Theta^*) <\infty.$
\end{proof}

We now impose a Lipschitzian requirement on the gradient map $\nabla_x
f(x;\theta)$  in $\theta$ uniformly in $x$.

\begin{assumption} \label{assump:weak_sharp_lipsch} There is a constant $M_{f,\theta}$ such that for
$\|\nabla_x f(x,\theta)-\nabla_x f(x,\theta^*)\|\leq
M_{f,\theta}\mbox{dist}(\theta,\Theta^*)$ for all $\theta \in \Theta$,
	$\theta^*\in \Theta^*$ and $x\in X$.
\end{assumption}
\begin{theorem}[{\bf Diminishing steplength scheme}]
Let Assumptions ~\ref{assump:differentiability},~\ref{assump:f_g_lips_contin_gradient},~\ref{assump:compact and
		lips cont}, and ~\ref{assump:weak_sharp_lipsch} hold and
		$\Theta^*$ is weak sharp. Let $\{x_k,\theta_k\}$ be the sequence
		generated by Algorithm ~\ref{alg1}.{ Additionally, let
		$\gamma_{g,k}$ be defined based on Assumption
			~\ref{assump:step_length} and $\gamma_{f,k}=\gamma_{g,k}$ for all $k>0$}.
Then, $\{x_k\}$ converges to a point in $X^*$ and $\{\theta_k\}$
converges to a point in $\Theta^*$ as $k \to \infty$.
\end{theorem}
\begin{proof}
By the nonexpansivity property of the Euclidean projection operator,
   for all $k>0$ {and any $x^* \in X^*$}, $\|x_{k+1}-{x^*}\|^2$ can be bounded as follows:
\begin{align*}
\|x_{k+1}-x^*\|^2&=\| \Pi_X ( x_k - \gamma_{f,k}
		\nabla_xf(x_k,\theta_k))-\Pi_X(x^*)\|^2\notag\\
  &\leq \| (x_k-x^*) - \gamma_{f,k} \nabla_xf(x_k,\theta_k)\|^2\\
  &=\|x_k-x^*\|^2-2\gamma_{f,k}\nabla_xf(x_k,\theta_k)^T(x_k-x^*)+\gamma_{f,k}^2\|\nabla_xf(x_k,\theta_k)\|^2\\
  &\leq\|x_k-x^*\|^2-2\gamma_{f,k}\nabla_xf(x_k,\theta^*)^T(x_k-x^*)-2\gamma_{f,k}r_{k}^T(x_k-x^*)+\gamma_{f,k}^2\|\nabla_xf(x_k,\theta_k)\|^2,
\end{align*}
where $r_k\triangleq\nabla_xf(x_k,\theta_k)-\nabla_xf(x_k,\theta^*)$. By leveraging convexity and the gradient inequality, we have that
\begin{align*}
f(x^*,\theta^*) \geq f(x_k,\theta^*)+\nabla(f(x_k,\theta^*)^T(x^*-x_k),
\end{align*}
implying that $-\nabla_xf(x_k,\theta^*)^T(x_k-x^*)\leq
   -(f(x_k,\theta^*)-f(x^*,\theta^*)).$
By the previous observation and the Cauchy-Schwartz inequality, we have the following:
\begin{align}
& \ \quad \|x_{k+1}-x^*\|^2 \notag\\&\leq
\|x_k-x^*\|^2-2\gamma_{f,k}(f(x_k,\theta^*)-f(x^*,\theta^*))-2\gamma_{f,k}r_{k}^T(x_k-x^*)+\gamma_{f,k}^2\|\nabla_xf(x_k,\theta_k)\|^2
\notag
\\ \notag
& \leq
\|x_k-x^*\|^2-2\gamma_{f,k}(f(x_k,\theta^*)-f(x^*,\theta^*)){+}2\gamma_{f,k}\|r_{k}\|\|x_k-x^*\|+\gamma_{f,k}^2\|\nabla_xf(x_k,\theta_k)\|^2\\
  & \leq
\|x_k-x^*\|^2-2\gamma_{f,k}(f(x_k,\theta^*)-f(x^*,\theta^*))+4CM_{f,\theta}\gamma_{f,k}
\mbox{dist}(\theta_k,\Theta^*)+\gamma_{f,k}^2\|\nabla_xf(x_k,\theta_k)\|^2\label{proof:ws_eq1},
\end{align}
{where $C$ is the constant in Assumption ~\ref{assump:compact and  lips cont}(a)}.
By Lemma ~\ref{lemma:rate_ws},
$$\sum_{k=1}^\infty\gamma_{f,k}\mbox{dist}(\theta_k,\Theta^*)=\sum_{k=1}^\infty\gamma_{g,k}\mbox{dist}(\theta_k,\Theta^*)<
\infty.$$ In addition, $$\sum\limits_{k=1}^\infty\gamma_{f,k}^2\|\nabla_xf(x_k,\theta_k)\|^2\leq C_{2}^2 \sum\limits_{k=1}^\infty\gamma_{f,k}^2< \infty,$$ where $C_2:=\sup_{x\in X, \theta \in \Theta} \|\nabla_x f(x,\theta)\|$. Hence, the conditions of Lemma~\ref{polyak2} are satisfied and the sequence $\|x_{k+1}-x^*\|$ is convergent for any $x^*\in X^*$ and $\sum_{k=1}^\infty \gamma_{f,k}(f(x_k,\theta^*)-f(x^*,\theta^*))< \infty.$
The the latter implies $\liminf_{k\to\infty}
(f(x_k,\theta^*)-f(x^*,\theta^*))=0$ {in view of
$\sum_{k=1}^\infty\gamma_{f,k}=\infty.$}
Since the set $X$ is closed, all accumulation points of $\{x_k\}$ lie in
$X$. Furthermore, since $f(x_k,\theta^*)\to f(x^*,\theta^*)$ along a
subsequence, by continuity of $f$ it follows that $\{x_k\}$ has a
subsequence converging to some point in $X^*$. Moreover, since
$\|x_k-x^*\|$ is { a convergent sequence}, the entire sequence
$\{x_k\}$ converges to some point in $X^*$. {Finally, the sequence
$\{\theta_k\}$ converges to a $\theta^* \in \Theta^*$, a consequence of
Lemma~\ref{lemma:rate_ws}.}
\end{proof}
\subsection{Nonsmooth convex optimization}\label{subsec:nonsmoothConvOpt}
In this section, we derive the global convergence and rate statements
for the regime when function $f(x;\theta)$ is not necessarily
differentiable.  Note that Assumptions ~\ref{assump:differentiability},
	~\ref{assump:f_g_lips_contin_gradient} and
	~\ref{assump:f_and_g_strconv} still hold for function $g$ and {for
		clarity, we restate them in the following assumption} and
		proceed to present a subgradient-based analog of
	Algorithm~\ref{alg1}.

\begin{assumption}
\label{assump: g_assump_sub_grad}
The function $g$ is continuously differentiable in $\theta$, strongly convex, and the gradient map $\nabla_{\theta} g(\theta)$ is Lipschitz continuous in $\theta$ with constant $G_{g}$.
\end{assumption}
\begin{algorithm}[{\bf Joint subgradient scheme}]\label{alg2}
Given an $x_0 \in X$ and a $\theta_0 \in \Theta$ and sequences
	$\{\gamma_{f,k},\gamma_{g,k}\}$, then
\begin{align*}
x_{k+1} & := \Pi_X \left( x_k - \gamma_{f,k} d_k\right),  \ \quad\quad\quad
	   & 	\forall k \geq 0,  \tag{nsOpt$(\theta_k)$}\\
\theta_{k+1} & := \Pi_{\Theta} \left( \theta_k - \gamma_{g,k}
	\nabla_{\theta} g(\theta_k)\right), \ \quad
		   & \forall k \geq 0, \tag{Learn}
\end{align*}
where  $d_k\in \partial f(x_k,\theta_k)$.
\end{algorithm}

\noindent We now state two assumptions employed in this
subsection, the first of which pertains to subgradient boundedness while
the second imposes Lipschitz continuity of $f(x,\theta)$ in $\theta$
uniformly in $x$.
\begin{assumption}[{\bf Subgradient boundedness}]
\label{assump:sub_grad_bound}
There exists an $M>0$ such that $\|d_k\|\leq M$ for all $d_k\in \partial
f(x_k,\theta_k)$ and for all $\theta_k \in \Theta$.
\end{assumption}
\begin{assumption}
\label{assump:sub_lips_cont}
There exists a constant $L_{f,\theta}$ such that
$|f(x,\theta_1)-f(x,\theta_2)|\leq L_{f,\theta}\|\theta_1-\theta_2\|
\quad \forall \theta_1,\theta_2 \in \Theta, x\in X.$
\end{assumption}
The following Lemma will be used subsequently in our convergence analysis.
\begin{lemma}
\label{lem:sub_main}
Let Assumptions ~\ref{assump:sub_grad_bound} and ~\ref{assump:sub_lips_cont} hold. Let $\{x_k\}$ and $\{\theta_k\}$ be the sequences generated by Algorithm ~\ref{alg2}. Then, for all $y\in X$ and $k>0$, we have
\begin{align*}
\|x_{k+1}-y\|^2  \leq  \|x_k-y\|^2-2\gamma_{f,k}(f(x_k,\theta^*)-f(y,\theta^*)) +4L_{f,\theta}\gamma_{f,k}\|\theta_k-\theta^*\|+ \gamma_{f,k}^2M^2,
\end{align*}
where $M$ is defined in Assumption ~\ref{assump:sub_grad_bound} and $L_{f,\theta}$ is the Lipschitz constant in Assumption ~\ref{assump:sub_lips_cont}.
\end{lemma}
\begin{proof}
By nonexpansivity of the Euclidean projector and triangle inequality, we
may bound $\|x_{k+1} - y\|$ as follows:
\begin{align*}
\|x_{k+1}-y\|^2&\leq \|\Pi_X(x_k-\gamma_{f,k}d_k)-{\Pi_X(y)}\|^2
\leq \|x_k-\gamma_{f,k}d_k-y\|^2 \\
&= \|x_k-y\|^2-2\gamma_{f,k}(x_k-y)^Td_k + \gamma_{f,k}^2\|d_k\|^2 \\
& \leq   \|x_k-y\|^2-2\gamma_{f,k}(x_k-y)^Td_k + \gamma_{f,k}^2M^2.
\end{align*}
Now, by leveraging convexity of function $f(x,\theta)$ in $x$ for all
$\theta$, we obtain
\begin{align}
\|x_{k+1}-y\|^2  &\leq  \|x_k-y\|^2-2\gamma_{f,k}(f(x_k,\theta_k)-f(y,\theta_k))+ \gamma_{f,k}^2M^2.\label{lem:sub_main ineq1}
\end{align}
 By Assumption~\ref{assump:sub_lips_cont}, the function $f(x,\theta)$ is Lipschitz
 continuous in $\theta$ for every $x$. Consequently,
 $|f(x_k,\theta_k)-f(x_k,\theta^*)|\leq
 L_{f,\theta}\|\theta_k-\theta^*\|$ and
 $|f(y,\theta_k)-f(y,\theta^*)|\leq L_{f,\theta}\|\theta_k-\theta^*\|$.
 It follows that
$$ f(x_k,\theta^*)-f(x_k,\theta_k)\leq
 L_{f,\theta}\|\theta^*-\theta_k\| \mbox{ and }
 f(y,\theta_k)-f(y,\theta^*) \leq L_{f,\theta}\|\theta_k-\theta^*\|.$$
 By combining these two inequalities, we get the following lower bound:
\begin{align*}
f(x_k,\theta_k)-f(y,\theta_k)\geq f(x_k,\theta^*)-f(y,\theta^*) -2L_{f,\theta}\|\theta_k-\theta^*\|.
\end{align*}
Now by combining above inequality with ~\eqref{lem:sub_main ineq1}, we have that
\begin{align}
\|x_{k+1}-y\|^2  \leq  \|x_k-y\|^2-2\gamma_{f,k}(f(x_k,\theta^*)-f(y,\theta^*)) +4L_{f,\theta}\gamma_{f,k}\|\theta_k-\theta^*\|+ \gamma_{f,k}^2M^2.\label{lemm:sub_ineq}
\end{align}
\end{proof}

\noindent {By leveraging Lemma~\ref{lem:sub_main}, we now provide the main convergence result for subgradient-based schemes for resolving misspecified convex optimization problems.}
\begin{proposition}[{\bf Global convergence for diminishing steplength
schemes}]
Let Assumptions { ~\ref{assump: g_assump_sub_grad}}, ~\ref{assump:sub_grad_bound}, and
	~\ref{assump:sub_lips_cont} hold. {Additionally, let
		$\gamma_{f,k}$ be defined based on Assumption
			~\ref{assump:step_length} and $\gamma_{g,k}$ be fixed at $\gamma_g$ so that $0<\gamma_g < 2/G_g$.} Let $\{x_k,\theta_k\}$ be the
	sequences generated by Algorithm ~\ref{alg2}. Then, $\{x_k\}$
	converges to a point in $X^*$ and $\{\theta_k\}$ converges to
{$\theta^*\in \Theta$}.
\end{proposition}
\begin{proof}
Using ~\eqref{lemm:sub_ineq} for $y=x^*$, {where $x^*$ is any point in $ X^*$}, we obtain
\begin{align*}
\|x_{k+1}-x^*\|^2  &\leq  \|x_k-x^*\|^2-2\gamma_{f,k}(f(x_k,\theta^*)-f(x^*,\theta^*)) + 4L_{f,\theta}\gamma_{f,k}\|\theta_k-\theta^*\|+ \gamma_{f,k}^2M^2.
\end{align*}
To prove the convergence, we employ Lemma ~\ref{polyak2}. Since
$\|\theta_k-\theta^*\| \leq q_g^k\|\theta_0-\theta^*\|$, we have that
\begin{align*}
\sum_{k=0}^\infty4L_{f,\theta} \gamma_{f,k}\|\theta_k-\theta^*\|
\leq \frac{4L_{f,\theta}\|\theta_0-\theta^*\|}{1-q_g} <\infty
\mbox{ and } \sum_{k=0}^\infty \gamma_{f,k}^2M^2 < \infty.
\end{align*}
 Hence, conditions of Lemma ~\ref{polyak2} are satisfied and $x_k\to
\bar{x}\in X$ as $k\to\infty$ and $\sum_{k=0}^\infty
\gamma_{f,k}(f(x_k,\theta^*)-f(x^*,\theta^*))<\infty$. Because
$\sum_{k=0}^\infty \gamma_{f,k}=\infty$, we can conclude that
$\liminf_{k \to \infty} f(x_k,\theta^*)= f(x^*,\theta^*)$. This
	implies that a subsequence of $\{x_k\}$ converges to a point in
		$X^*$. But the entire sequence is convergent, implying that the
		entire sequence converges to a point in $X^*$.
	Furthermore, $\theta_k \to \theta^*$ as $k \to \infty$.
\end{proof}

{In keeping with the focus of this paper, we now provide derive rate
statements for the function iterates where we quantify the impact of
learning.  }
\begin{proposition}[{\bf Rate analysis with averaging}]\label{prop:sub_aver}
Let Assumptions {~\ref{assump: g_assump_sub_grad}} , ~\ref{assump:sub_grad_bound}, and
	~\ref{assump:sub_lips_cont} hold.
Let  $\gamma_{g,k}$ be fixed at $\gamma_g$ such that $0<\gamma_{g}<
2/G_g$.  Consider the sequence $\{x_k,\theta_k\}$ generated by
Algorithm ~\ref{alg2}  and $\bar x_{k}\triangleq\frac{\sum_{i=0}^k
	\gamma_{f,i}x_i}{\sum_{i=0}^k\gamma_{f,i}}$. Then the following hold:
\begin{enumerate}
\item[(i)] If
		$\gamma_{f,k}$ is defined based on Assumption
			~\ref{assump:step_length}, then
     $$\lim_{k\to\infty} |f(\xbar_{k},\theta_{k})-f(x^*,\theta^*)|=0.$$
\item[(ii)] Suppose Algorithm~\ref{alg2} is to be terminated after
	$K$ iterations and $\gamma_f$ (the optimal constant steplength) is
		defined as \begin{align}\gamma_{f,K}{=} \frac{\|x_0-x^*\|}{M\sqrt{{K}+1}},\label{opt_step_rule}\end{align} then
$$ |f(\xbar_{K},\theta_{K})-f(x^*,\theta^*)| \leq
	\frac{d_x}{\sqrt{K+1}} + \|\theta_0-\theta^*\|\left(L_{f,\theta} q_g^K +
		\frac{c_{\theta}}{(K+1)}\right), $$
where $d_x = M\|x_0 - x^*\|$ and
$c_{\theta} = 2L_{f,\theta}/(1-q_g).$
\end{enumerate}
\end{proposition}
\begin{proof}
(i) By letting $y=x^*$ in ~\eqref{lemm:sub_ineq}  and by summing
~\eqref{lemm:sub_ineq} over {$k$}, we have that the following holds:
\begin{align*}
\|x_{k+1}-x^*\|^2  & \leq
	\|x_0-x^*\|^2-2\sum_{i=0}^k\gamma_{f,i}(f(x_i,\theta^*)-f(x^*,\theta^*))
	 +4L_{f,\theta}\sum_{i=0}^k\gamma_{f,i}\|\theta_i-\theta^*\|+M^2\sum_{i=0}^k \gamma_{f,i}^2.
\end{align*}
By the nonnegativity of $\|x_{k+1}-x^*\|^2$, it follows that
\begin{align}
 2\sum_{i=0}^k\gamma_{f,i}(f(x_i,\theta^*)-f(x^*,\theta^*))\leq
	\|x_0-x^*\|^2+4L_{f,\theta}\sum_{i=0}^k\gamma_{f,i}\|\theta_i-\theta^*\| +M^2\sum_{i=0}^k \gamma_{f,i}^2.\label{prop:sub_summ}
\end{align}
From the convexity of $f(x,\theta^*)$ in $x$, we have the following:
\begin{align}
 \frac{2}{\sum_{i=0}^k\gamma_{f,i}}\sum_{i=0}^k\gamma_{f,i}(f(x_i,\theta^*)-f(x^*,\theta^*))  \geq 2 \left( f( \bar x_k,\theta^*
				) -f(x^*,\theta^*)\right).
\label{prop:sub_conv}
\end{align}
By combining ~\eqref{prop:sub_summ} and ~\eqref{prop:sub_conv}, we obtain the inequality
\begin{align*}
 f\left( \bar x_k,\theta^* \right) -f(x^*,\theta^*) \leq \frac{\|x_0-x^*\|^2+M^2\sum_{i=0}^k
	 \gamma_{f,i}^2}{2\sum_{i=0}^k\gamma_{f,i}}
 +\frac{2L_{f,\theta}\sum_{i=0}^k\gamma_{f,i}\|\theta_i-\theta^*\|}{\sum_{i=0}^k\gamma_{f,i}}.
\end{align*}
Notably, the second term arises from learning and can be further bounded
	as follows:
\begin{align*}
2L_{f,\theta}\sum_{i=0}^k\gamma_{f,i}\|\theta_i-\theta^*\| & \leq 2L_{f,\theta} \gamma_{f,0}
\|\theta_0-\theta^*\| \sum_{i=0}^k q_{g}^i\leq \frac{2L_{f,\theta} \gamma_{f,0}
\|\theta_0-\theta^*\| (1-q_g^{k+1})}{1-q_g}.
\end{align*}
Consequently, we may bound $f(\bar x_k,\theta^*)-f(x^*,\theta^*)$ as
	follows:
\begin{align*}
\quad f\left( \xbar_k,\theta^* \right) -f(x^*,\theta^*)\leq \frac{\|x_0-x^*\|^2+M^2\sum_{i=0}^k
	 \gamma_{f,i}^2}{2\sum_{i=0}^k\gamma_{f,i}} +  \frac{2L_{f,\theta} \gamma_{f,0}
\|\theta_0-\theta^*\| (1-q_g^{k+1})}{(1-q_g)\sum_{i=0}^k\gamma_{f,i}}.
\end{align*}
It follows that $|f(\xbar_k,\theta_k)-f(x^*,\theta^*)|$ may be bounded
	as follows:
\begin{align*}
  |f(\xbar_k,\theta_k)-f(x^*,\theta^*)|& \leq |f(\xbar_k,\theta_k)-f(\xbar_k,\theta^*)|
	+ |f(\xbar_k,\theta^*)-f(\xbar_k,\theta^*)|\\
		& \leq L_{f,\theta} q_g^k \|\theta_0-\theta^*\|
		 + \frac{\|x_0-x^*\|^2+M^2\sum_{i=0}^k
	 \gamma_{f,i}^2}{2\sum_{i=0}^k\gamma_{f,i}}
 +  \frac{2L_{f,\theta} \gamma_{f,0}
\|\theta_0-\theta^*\| (1-q_g^{k+1})}{(1-q_g)\sum_{i=0}^k\gamma_{f,i}}.
\end{align*}
Since $q_g < 1$, $\sum_{i=0}^{\infty} \gamma_{f,i} = \infty$, and
$\sum_{i=0}^{\infty} \gamma^2_{f,i} < \infty$, it
	follows that
$ \lim_{k\to \infty} |f(\xbar_k,\theta_k)-f(x^*,\theta^*)| = 0. $ \\
(ii) Next, if we assume that the steplength is fixed at $\gamma_f$,
after $k = K$ iterations, the bound on the error is given by
the following:
\begin{align*}
 \ |f(\bar x_{K},\theta_{K})-f(x^*,\theta^*)|& \leq L_{f,\theta} q_g^{K} \|\theta_0-\theta^*\|
		 + \frac{\|x_0-x^*\|^2+M^2 (K+1) \gamma_f^2}{2(K+1) \gamma_f} +  \frac{2L_{f,\theta}
\|\theta_0-\theta^*\| (1-q_g^{K+1})}{(1-q_g)(K+1)}.
\end{align*}
 If we minimize the right hand side with respect to
	$\gamma_f$, we arrive at the best optimal constant stepsize
\begin{align*}
\gamma_{f,K}{=} \frac{\|x_0-x^*\|}{M\sqrt{K+1}}.
\end{align*}
Using this step length, we have the optimal convergence rate of
\begin{align*}
|f(\xbar_{K},\theta_{K})-f(x^*,\theta^*)|  &\leq  L_{f,\theta} q_g^{K}
  \|\theta_0-\theta^*\| +  \frac{2L_{f,\theta}
\|\theta_0-\theta^*\| (1-q_g^{K+1})}{(1-q_g)(K+1)}  +
\frac{M\|x_0-x^*\|}{\sqrt{ K+1}} \\
	& \leq
	\frac{d_x}{\sqrt{K+1}} + \|\theta_0-\theta^*\| \left( L_{f,\theta}
			q_g^K + \frac{2L_{f,\theta}}{(1-q_g)(K+1)}\right) \\
	& =
	\frac{d_x}{\sqrt{K+1}} + \underbrace{ \|\theta_0-\theta^*\|\left(L_{f,\theta} q_g^K +
		\frac{c_{\theta}}{(K+1)}\right)}_{\tiny \mbox{Impact from
			learning}},
\end{align*}
where $d_x = M\|x_0 - x^*\|$ and
$c_{\theta} = 2L_{f,\theta}/(1-q_g).$
\end{proof}

\noindent {\bf Remark:}  Standard subgradient methods for convex
optimization display a convergence rate of ${\cal O}(1/\sqrt{K})$ in
function value ~\cite{Boyd:2004:CO:993483}. Notably, the joint scheme shows {\bf no} degradation in
the rate, not even in a constant factor sense. More specifically,
	the modification in the rate is given by
		$\|\theta_0-\theta^*\|{\cal O}\left(\frac{1}{K} + q^K\right)$, with both terms arising from learning diminishing to zero at a
faster rate. This factor is scaled by the distance of $\theta_0$ from
its true value $\theta^*$ and we recover the original rate if $\theta_0
= \theta^*$.

\section{Misspecified monotone variational inequality problems}\label{sec:monotone VI}
{In the problem formulation investigated thus far, the misspecified
	parameter $\theta^*$ lay in the objective function $f$. Yet
		in many instances, the misspecification may also arise in the
		constraint set. In particular, consider the following misspecified problem $(\cal C'(\theta^*))$, defined as
\begin{align}
\tag{$\cal C'(\theta^*)$} \min_{x \in X(\theta^*)} \, f(x,\theta^*), \end{align} where
$x \in \Real^n$, $f: X \times \Theta \to \Real$ is a convex function in
$x$ for every $\theta \in \Theta \subseteq \Real^m$. One approach is to
relax the constraints that are misspecified and consider a Lagrangian
(or an augmented Lagrangian) approach.  Another approach lies in leveraging the convexity of the problem and considering the
	complementarity problem arising from the first-order
		(sufficient) optimality
		conditions. It is well known that if the constraints
set $X(\theta^*)$ has an algebraic structure given by
$$X(\theta^*) \triangleq \left\{x:
h(x;\theta^*)\geq 0\, ,x\geq 0 \right\},$$ where $h(x,\theta)$ is a
convex function in $x$ for every $\theta$, then the first-order
conditions are given by
\begin{align}
\tag{CP$(\theta$)}\begin{aligned}
		0 \leq x & \perp \nabla_x f(x,\theta) - \nabla_x h(x,\theta)^T
			\lambda \geq 0, \\
		0 \leq \lambda & \perp h(x,\theta) \geq 0,
	\end{aligned} \end{align}
where $u \perp v \equiv [u]_i[v]_i = 0$ for every $i$. It is well known
~\cite{facchinei02finite} that this complementarity problem
(CP$(\theta)$) is equivalent to VI$(Z,F(.;\theta))$, where $Z \triangleq
\Real^{m+n}_+$ and $F(z)$,  defined as
$$ F(z) \triangleq \pmat{ \nabla_x f(x,\theta) - \nabla_x h(x,\theta)^T
			\lambda  \\
		 h(x,\theta)}, $$ is a monotone map. More generally,
variational inequality problems represent a broadly encompassing tool
	for capturing a range of equilibrium problems arising in economics,
		engineering, and applied sciences (cf.~\cite{Pang03I}). 	This
motivates us to extend the realm of computational problems to
accommodate the class of misspecified monotone variational inequality
problems, which is formally defined later in this section. By doing so, we may not only accommodate the problem $(\cal C'(\theta^*))$, but
		also we can consider a far broader class of misspecified problems.
		
Given a set $X \subseteq \Real^n$ and $F:X \to \Real^n$, a single-valued
mapping, then a variational inequality problem VI$(X,F)$ requires an $x
\in X$ such that $(y-x)^TF(x) \geq 0$ for all $y \in X.$
More specifically, we consider the misspecifed variational inequality
problem VI$(X,F(\bullet;\theta^*))$ where $F:X \times \Theta \to
\Real^n$:
\begin{align}
\tag{$\cal V(\theta^*)$} (y-x)^TF(x;\theta^*) \geq 0, \qquad \forall y \in X. \end{align}
In Subsections ~\ref{subsec: Extra grad
	VI} and ~\ref{subsecc:Tikhonov VI }, we present extragradient and regularized first-order schemes,
	respectively, for misspsecified monotone variational inequality problems with strongly convex learning problems. Throughout this section, we make the following assumption on the learning function $g$ and map $F$.
	\begin{assumption} \label{assump:F lips cont}
\be
\item[(a)] The function $g$ is differentiable, strongly convex with constant $\eta_g$, and Lipschitz continuous in gradient with constant $G_g$.
\item[(b)] The map $F$ is monotone in $x$ and uniformly Lipschitz continuous in $x$ and $\theta$ with constants $L_{F,x}$ and $L_{F,\theta}$, respectively:
    \begin{align*}
	\|F(x_1;\theta)-F(x_2;\theta)\| & \leq L_{F,x}\|x_1-x_2\|\quad \forall
	x_1,x_2\in X, \quad \forall \theta\in \Theta, \\
           \|F(x,\theta_1)-F(x,\theta_2)\| & \leq
		   L_{F,\theta}\|\theta_1-\theta_2\|\quad \ \forall
		   \theta_1,\theta_2\in \Theta,  \quad \ \forall x\in X.
\end{align*}
\ee
\end{assumption}
\subsection{Extragradient schemes}\label{subsec: Extra grad VI}
{The extragradient scheme was first proposed by
Korpolevich~\cite{Korpelevich} and such approaches have been enormously
	useful in the solution of both convex optimization problems and
	monotone variational inequality problems~\cite{Pang03I} via constant
	steplength schemes.  Subsequently,
			   Nemirovski~\cite{Nemirovski:2005:PRC:1029075.1039910}
proposed a prox-type method with a general distance function with
	convergence rate of ${\cal O}(1/K)$,  which is equivalent to
	extragradient scheme under a  Euclidian distance function}.  In this
	subsection, we consider whether the extragradient framework can be
	extended to the regime of interest and propose a misspecified
	variant of the extragradient scheme:
\begin{algorithm}[{\bf A joint extragradient scheme}]
\label{alg4}
Given $x_0 \in X,$ $\theta_0 \in \Theta$ and a steplength $\tau$,
\begin{align}
z_{k+1} &:=\Pi_X(x_k-\tau F(x_k;\theta_k))\qquad  & \forall k>0,
	\tag{Extra$_x(\theta_k)$}\\
x_{k+1} &:=\Pi_X(x_k-\tau F(z_{k+1};\theta_k))\quad & \forall k>0,
	\tag{Extra$_z(\theta_k)$}\\
\theta_{k+1} &:=\Pi_{\Theta}(\theta_k-\gamma_g \nabla_{\theta}
		g(\theta_k))  \ \quad \quad & \forall k>0. \tag{Learn}
\end{align}
\end{algorithm}
{Unlike the standard projected gradient framework, the extragradient}
scheme requires two consecutive gradient steps with the same belief
$\theta_k$. {Note that the proof of convergence follows along the lines
	of that provided by Facchinei and Pang~\cite{Pang03II}, but with
		some care required to handle  the extra terms  arising from
		learning. We  begin  by presenting a supporting Lemma.}
\begin{lemma}
\label{lemma:extra bound}
Let Assumption ~\ref{assump:F lips cont} holds and $\{x_k,\theta_k\}$ be
the sequence generated by ~{Algorithm} ~\ref{alg4}. If $x^*$ is a point
in $X^*$, then for all $k$,
$$\|x_{k+1}-x^*\|^2\leq \|x_k-x^*\|^2-(1-\tau^2L_{F,x}^2)\|z_{k+1}-x_k\|^2+2\tau L_{F,\theta}\|\theta_{k+1}-\theta^*\|\|x^*-z_{k+1}\|.$$
\end{lemma}
\begin{proof}
{By the projection property, we have that for any $x\in \R^n$,}
\begin{align*}
\|\Pi_X(x)-z\|^2\leq \|x-z\|^2-\|\Pi_X(x)-x\|^2 \quad \mbox{for all} \quad z\in X.
\end{align*}
Using above relation with $x=x_k-\tau F(z_{k+1};\theta_k)$ and $z=x^*$, we obtain
\begin{align*}
\|x_{k+1}-x^*\|^2\leq \|x_k-\tau F(z_{k+1};\theta_k)-x^*\|^2-\|x_{k+1}-(x_k-\tau F(z_{k+1};\theta_k))\|^2.
\end{align*}
By expanding the terms on the right hand side, we have
\begin{align}
& \quad \ \|x_{k+1}-x^*\|^2\\
&\leq \|x_k-x^*\|^2-\|x_{k+1}-x_k\|^2+2\tau F(z_{k+1};\theta_k)^T(x^*-x_{k+1})\notag\\
&\leq \|x_k-x^*\|^2-\|x_{k+1}-x_k\|^2+2\tau F(z_{k+1};\theta^*)^T(x^*-x_{k+1})\notag\\ &+ 2\tau (F(z_{k+1};\theta_k)-F(z_{k+1},\theta^*))^T(x^*-x_{k+1})\notag\\
&\leq \|x_k-x^*\|^2-\|x_{k+1}-x_k\|^2+2\tau F(z_{k+1};\theta^*)^T(x^*-x_{k+1})+ 2\tau r_{k+1}^T(x^*-x_{k+1}),\label{proof:extra_ineq1}
\end{align}
where {the second inequality is a consequence of adding and
	subtracting $F(z_{k+1},\theta^*)^T(x^*-x_{k+1})$ and $r_{k+1}$ is
		defined as }  $r_{k+1}\triangleq
		F(z_{k+1},\theta_k)-F(z_{k+1},\theta^*)$. By {the monotonicity of
		$F(\bullet;\theta^*)$} {over} $X$, it follows that
\begin{align*}
(F(z_{k+1},\theta^*)-F(x^*,\theta^*))^T(z_{k+1}-x^*)\geq 0,
\end{align*}
and since $x^*\in {X^*}$, the above inequality can be simplified to
$ F(z_{k+1},\theta^*)^T(z_{k+1}-x^*)\geq 0.$
Hence, by adding and subtracting $x_{k+1}$ in the above inequality, we obtain that
$$F(z_{k+1};\theta^*)^T(z_{k+1}-x_{k+1})+F(z_{k+1};\theta^*)^T(x_{k+1}-x^*)\geq 0,$$
which implies
$$ F(z_{k+1};\theta^*)^T(z_{k+1}-x_{k+1})\geq F(z_{k+1};\theta^*)^T(x^*-x_{k+1}).$$
Using this relation in ~\eqref{proof:extra_ineq1}, we see that
\begin{align*}
\|x_{k+1}-x^*\|^2&\leq \|x_k-x^*\|^2-\|x_{k+1}-x_k\|^2+2\tau F(z_{k+1},\theta^*)^T(z_{k+1}-x_{k+1}) + 2\tau r_{k+1}^T(x^*-x_{k+1})\\
&\leq \|x_k-x^*\|^2-\|x_{k+1}-x_k\|^2-2\tau F(z_{k+1},\theta^*)^T(x_{k+1}-z_{k+1}) + 2\tau r_{k+1}^T(x^*-x_{k+1}).
\end{align*}
By writing $x_{k+1}-x_k=(x_{k+1}-z_{k+1})+(z_{k+1}-x_k)$, we can expand $\|x_{k+1}-x_k\|^2$ as follow:
\begin{align*}
\|x_{k+1}-x_k\|^2&=\|(x_{k+1}-z_{k+1})+(z_{k+1}-x_k)\|^2\\
&=\|x_{k+1}-z_{k+1}\|^2+\|z_{k+1}-x_k\|^2-2(x_k-z_{k+1})^T(x_{k+1}-z_{k+1}).
\end{align*}
By combining the terms in the inner product with $x_{k+1}-z_{k+1}$, we obtain
\begin{align}
\|x_{k+1}-x^*\|^2 &\leq \|x_k-x^*\|^2-\|x_{k+1}-z_{k+1}\|^2-\|z_{k+1}-x_k\|^2\notag\\&+2(x_{k+1}-z_{k+1})^T(x_k-\tau F(z_{k+1},\theta^*)-z_{k+1})+2\tau r_{k+1}^T(x^*-x_{k+1}).\label{proof:extra ineq2}
\end{align}
Through the addition and subtraction of terms,
 $(x_{k+1}-z_{k+1})^T(x_k-\tau
		F(z_{k+1},\theta^*)-z_{k+1})$ as follows:
\begin{align*}
(x_{k+1}-z_{k+1})^T(x_k-\tau F(z_{k+1},\theta^*)-z_{k+1}) &
=(x_{k+1}-z_{k+1})^T(x_k-\tau F(x_k,\theta_k)-z_{k+1}) \\ &+\tau
(x_{k+1}-z_{k+1})^T(F(x_k,\theta_k)-F(z_{k+1},\theta_k)) \\ &+\tau(x_{k+1}-z_{k+1})^T(F(z_{k+1},\theta_k)-F(z_{k+1},\theta^*)).
\end{align*}
{Since $x_{k+1}\in X$ and $z_{k+1}=\Pi_X(x_k-\tau F(x_k,\theta_k))$, the
first term on the right hand side is nonpositive by the projection
property. By leveraging this property and the Lipschitz continuity} of $F(\bullet,\theta^*)$ in $x$, we have
\begin{align}
& \ \quad(x_{k+1}-z_{k+1})^T(x_k-\tau F(z_{k+1},\theta^*)-z_{k+1})\notag\\
&\leq \tau (x_{k+1}-z_{k+1})^T(F(x_k,\theta_k)-F(z_{k+1},\theta_k))+\tau(x_{k+1}-z_{k+1})^T(F(z_{k+1},\theta_k)-F(z_{k+1},\theta^*))\notag\\
&\leq \tau L_{F,x}\|x_{k+1}-z_{k+1}\|\|x_k-z_{k+1}\|+\tau r_{k+1}^T(x_{k+1}-z_{k+1})\notag\\
&\leq {1\over 2}(\|x_{k+1}-z_{k+1}\|^2+\tau^2L_{F,x}^2\|x_k-z_{k+1}\|^2)+\tau r_{k+1}^T(x_{k+1}-z_{k+1}).\label{proof:extra ineq3}
\end{align}
From the Lipschitz continuity of $F(x,\theta)$ in $\theta$, {it
	follows that}
	$\|r_{k+1}\|=\|F(z_{k+1},\theta_k)-F(z_{k+1},\theta^*)\|\leq
	L_{F,\theta}\|\theta_k-\theta^*\|$. {By employing this bound and
		by substituting ~\eqref{proof:extra ineq3} in
			~\eqref{proof:extra ineq2}}, {the result follows.}
\begin{align*}
\|x_{k+1}-x^*\|^2&\leq
\|x_k-x^*\|^2-\|x_{k+1}-z_{k+1}\|^2-\|z_{k+1}-x_k\|^2+\|x_{k+1}-z_{k+1}\|^2+\tau^2{L_{F,x}}^2\|x_k-z_{k+1}\|^2\\&+2\tau \|r_{k+1}\|\|x^*-z_{k+1}\|\\
&=\|x_k-x^*\|^2-(1-\tau^2{L_{F,x}}^2)\|z_{k+1}-x_k\|^2+2\tau \|r_{k+1}\|\|x^*-z_{k+1}\|\\
&=\|x_k-x^*\|^2-(1-\tau^2L_{F,x}^2)\|z_{k+1}-x_k\|^2+2\tau L_{F,\theta}\|\theta_{{k}}-\theta^*\|\|x^*-z_{k+1}\|.
\end{align*}
\end{proof}

We now leverage this result in proving the convergence of the iterates
produced by Algorithm~\ref{alg4}.
\begin{theorem}[{\bf Convergence of extragradient scheme}]
Let Assumption ~\ref{assump:F lips cont} holds and {$\Theta$ is bounded}. In addition, assume that
stepsize $\gamma_{g,k}$ is fixed at $\gamma_g$, where $\gamma_g<
{2\over G_g}$. Let $\{x_k,\theta_k\}$ be the sequence generated by
Algorithm ~\ref{alg4} with {$$\tau^2  < {1\over
	L^2_{F,x}+2L_{F,\theta} \|\theta_0-\theta^*\|}.$$} Then $\{x_k\}$
converges to a point in $X^*$ and $\{\theta_k\}$ converges to {$\theta^*\in \Theta$} as $k\to \infty$.
\end{theorem}
\begin{proof}
From Lemma ~\ref{lemma:extra bound}, we have
$$\|x_{k+1}-x^*\|^2\leq \|x_k-x^*\|^2-(1-\tau^2L_{F,x}^2)\|z_{k+1}-x_k\|^2+2\tau L_{F,\theta}\|\theta_{{k}}-\theta^*\|\|x^*-z_{k+1}\|,$$
where $x^*$ is any point in  $X^*$. By writing
$x^*-z_{k+1}=(x_k-z_{k+1}) +  (x^*-x_k)$ and using the triangle inequality, we obtain that
\begin{align*}
\|x_{k+1}-x^*\|^2&\leq \|x_k-x^*\|^2-(1-\tau^2L_{F,x}^2)\|z_{k+1}-x_k\|^2+2\tau L_{F,\theta}\|\theta_{k}-\theta^*\|(\|x_k-z_{k+1}\|+\|x^*-x_k\|)\\
&\leq  \|x_k-x^*\|^2-(1-\tau^2L_{F,x}^2)\|z_{k+1}-x_k\|^2\\&+
L_{F,\theta}\|\theta_{{k}}-\theta^*\|(\tau^2\|x_k-z_{k+1}\|^2+\tau^2\|x^*-x_k\|^2 +
		2),
\end{align*}
 from {$2a\leq a^2 + {1}.$} By strong convexity of function $g$, there exist a constant $q_g\in(0,1)$ such that $\|\theta_{k}-\theta_0\|\leq q_g^{k-1}\|\theta_0-\theta^*\|$. By replacing this bound into the above inequality and then combining the similar terms, we get
{\begin{align}
& \quad \ \|x_{k+1}-x^*\|^2 \notag \\
&\leq
\|x_k-x^*\|^2-(1-\tau^2L_{F,x}^2)\|z_{k+1}-x_k\|^2+
L_{F,\theta}q_g^{k-1}\|\theta_0-\theta^*\|\left(\tau^2\|x_k-z_{k+1}\|^2+\tau^2\|x^*-x_k\|^2 +
		2\right)\notag\\
&\leq (1+\tau^2
		L_{F,\theta}\|\theta_0-\theta^*\|q_g^{k-1})\|x_k-x^*\|^2\notag-(1-\tau^2(L_{F,x}^2+
				L_{F,\theta}\|\theta_0-\theta^*\|q_g^{k-1}))\|z_{k+1}-x_k\|^2
	\notag	\\
&  +2L_{F,\theta}q_g^{k-1}\|\theta_0-\theta^*\| .\label{proof:extra ineq 4}
\end{align}}
To prove that the sequence $\{x_k\}$ converges to a point in $X^*$, we make use of Lemma ~\ref{polyak2}.  To check that conditions of Lemma are satisfied, we first see that
$$\sum_{k=1}^\infty\tau^2
		L_{F,\theta}\|\theta_0-\theta^*\|q_g^{k-1}  \leq \frac{{\tau^2}
			L_{F,\theta}\|\theta_0-\theta^*\|}{1-q_g}<\infty  \mbox{ and
				}
			{\sum_{k=1}^\infty2L_{F,\theta}q_g^{k-1}\|\theta_0-\theta^*\| \leq \frac{2L_{F,\theta}\|\theta_0-\theta^*\|}{1-q_g} <\infty.}$$
			In addition, $\tau$ satisfies the following for every $k$:
$$ \tau^2 <    \frac{1}{L_{F,x}^2 + L_{F,\theta} \|\theta_0-\theta^*\|}
\leq \frac{1}{L_{F,x}^2 + L_{F,\theta} \|\theta_0-\theta^*\| q_g^{k-1}}
. $$
Consequently, {$(1-\tau^2(L_{F,x}^2+
				L_{F,\theta}\|\theta_0-\theta^*\|q_g^{k-1})) > 0$} {for all $k>0$}.
Then, by Lemma ~\ref{polyak2}, we have that (i)  $\{\|x_k-x^*\|\}$ is a
convergent sequence and (ii) $\sum_{k=1}^\infty (1-\tau^2(L_{f,x}^2 +
		L_{f,\theta} \|\theta_0-\theta^*\| q_g^{k-1}))\|z_{k+1}-x_k\|^2 <
\infty$. By (i), $\{x_k\} \to \bar x$ as $k \to \infty$ where $\bar x$
is not necessarily a point in $X^*$. Since (ii) holds and by observing
that $\sum_{k=1}^\infty (1-\tau^2(L_{f,x}^2 +
		L_{f,\theta} \|\theta_0-\theta^*\| q_g^{k-1})) = \infty$, it follows
that $\liminf_{k \to \infty} \|z_{k+1}-x_k\| = 0.$  Consequently, we
have that for some subsequence ${\cal K}$,
\begin{align*}
 \bar x = \lim_{{\cal K} \ni k\to\infty} x_k= \lim_{{\cal K} \ni k\to\infty}
z_{k+1}=\lim_{{\cal K} \ni k\to\infty} \Pi_X(x_k-\tau
		F(x_k;\theta_k))=\Pi_X(\bar{x}-\tau F(\bar{x},\theta^*)).
\end{align*}
This implies that $\bar x$ is a point in $X^*$. But since $\{x_k\}$ is a
convergent sequence, the entire sequence converges to $\bar x$ and the
result follows.
\end{proof}

{\bf Remark:} It can be observed that if $\theta_0 = \theta^*$, then  we
recover the standard bound on the steplength for extragradient schemes. While we do not analyze the rate of extragradient schemes, we
	believe that analogous rate statements may be possible, akin to
		those provided by
		Nemirovski~\cite{Nemirovski:2005:PRC:1029075.1039910}.
\subsection{Regularized schemes for monotone VIs}\label{subsecc:Tikhonov VI }
{Consider a perfectly specified problem VI$(X,F)$, where $F$ is a
monotone map over a set $X \subseteq \Real^n$ and assume that $x^*$
denotes its least square norm solution. Consider the
$\epsilon$-regularized problem, denoted by
		VI$(X,F+\epsilon {\bf I})$, where $\epsilon$ is a positive
constant and ${\bf I}$ is an identity map. Since the
map $F+\epsilon I$ is strongly monotone as a consequence of the
regularization, VI$(X,F+\epsilon {\bf I})$ admits a unique
solution. This motivates the {\em exact} Tikhonov regularization
	method that generates a sequence $\{z_k\}$ where
$z_k$ solves VI$(K,F+ \epsilon_k {\bf I})$, $\epsilon_k$ denotes the
regularization at the k$^{th}$ iteration, and $\epsilon_k \to 0$ as $k
\to \infty$.  Under suitable conditions
(see ~\cite[Ch.12]{tikhonov63solution,tikhonov76methods,facchinei02finite}) the sequence $\left\{z_k
\right\}$ converges to $z^{*}$ as $\epsilon_k \rightarrow 0.$
The standard structure of the Tikhonov regularization scheme
requires obtaining exact or increasingly exact solutions of the
	subproblems VI$(X,F+\epsilon_k {\bf I})$, a relatively costly
	process. An alternative lies in taking a simple projected gradient
	step on the regularized map~\cite{Polyak87} and updating the
	regularization and steplength sequence at appropriate rates. This
	framework appears to have been first mentioned in
	~\cite{golshtein89modified} and further analyzed in
	~\cite{KannanS12_SIAM} and is often referred to as {\em iterative}
Tikhonov regularization and defined as follows:
$$ x_{k+1}:=\Pi_X\left(x_k- \gamma_{k}(F(x_k)+\epsilon_k x_k)\right) \quad  \forall k>0,$$
where $\gamma_k$ and $\epsilon_k$ are two vanishing sequences satisfying
certain requirements. The reader can refer to ~\cite{KannanS12_SIAM} for
more details. Inspired by this framework,  we introduce
a class of (Tikhonov) regularized schemes for the solution of
misspecified monotone variational inequality problems:

\begin{algorithm}[{\bf A regularized projection scheme}]
\label{alg3}
Given an $x_0 \in X$ and $\theta_0 \in \Theta$ and sequences
	$\{\gamma_{f,k}, {\gamma_{g,k}\}}$ and $\{\epsilon_k\}$,
\begin{align}
x_{k+1} &:=\Pi_X\left(x_k- \gamma_{f,k}(F(x_k,\theta_k)+\epsilon_k
			x_k)\right) \quad  & \forall k>0,
	\tag{Var$(\theta_k,\epsilon_k)$}\\
\theta_{k+1} &:=\Pi_{\Theta}\left(\theta_k-
		\gamma_{g,k}\nabla_{\theta}g(\theta_k)\right) \ \quad \qquad
\qquad & \forall k>0. \tag{Learn}
\end{align}
\end{algorithm}
{In our analysis}, we consider {two auxiliary} sequences $\{x^t_k\}$ and
$\{z^t_k\}$, defined as follows:
\begin{align}
x^t_k & :=\Pi_X(x^t_k- \gamma_{f,k}(F(x^t_k,\theta_k)+\epsilon_k x^t_k))\quad \forall k>0,
\tag{Tik$(\theta_k)$}\\
z^t_k & :=\Pi_X(z^t_k- \gamma_{f,k}(F(z^t_k,\theta^*)+\epsilon_k
			z^t_k)) \ \quad \forall k>0. \tag{Tik$(\theta^*)$}
\end{align}
Note that $\{x_k^t\}$ denotes the Tikhonov sequence associated with an
estimate of $\theta^*$, given by $\theta_k$, and each iterate represents
	the solution of the regularized problem
	VI$(X,F(\bullet;\theta_k)+\epsilon_k {\bf I})$. The iterate $x_k^t$
	can be viewed as a solution to a fixed-point problem, an alternative
	avenue for stating that $x_k^t$ is a solution of
	VI$(X,F(\bullet;\theta_k)+\epsilon_k {\bf I})$.  Analogously
	$\{z_k^t\}$ represents a  sequence in which each iterate is a
	solution to the regularized problem
	VI$(X,F(\bullet;\theta^*)+\epsilon_k {\bf I}).$ In what follows, we
		present a series of Lemmas that will be used to prove the
			convergence of the sequence $\{x_k\}$ to the least-norm solution of problem $\cal V(\theta^*)$. The proof sketch is as follows:
In Lemma~\ref{Lemma:tikh_1}, we relate $\{x^t_k\}$ with $\{z^t_k\}$ and show that as $\theta_k$ converges to $\theta^*$, $\{x^t_k\}$ converges to $\{z^t_k\}$. Lemmas ~\ref{Lemma:tikh_2}, ~\ref{Lemma:tikh_3} and ~\ref{Lemma:tikh_4}, when combined, show that as $k\to \infty$, the
iterative Tikhonov sequence $\{x_k\}$ converges to the sequence
$\{x_k^t\}$, by first deriving the bound on $\|x_k-x_k^t\|$ and then
showing that this bound goes to zero. Consequently, convergence of
$\{x_k\}$ to the least norm solution will be immediate since we know
that $\|x_k^t-z_k^t\| \to 0$ as $k \to \infty$ and  $\{z_k^t\}$
converges to the least norm solution of problem $\cal V(\theta^*)$. {We
make the following assumptions on the set $X$ and also on the stepsize and regularization sequences:
\begin{assumption}
\label{assump:VI_X_compact}
The set $X$ is compact and $\sup_{x\in X} \|x\|\leq M$, where M is a constant.
\end{assumption}}
\begin{assumption}
\label{assump:tikh stepsize}
The following hold:
\be
\item[(a)] $0<\gamma_{f,k}\leq \frac{\epsilon_k}{(L_{F,x}+\epsilon_k)^2}
{\leq \frac{\epsilon_0}{L_{F,x}^2}}$ for
all $k$;
\item[(b)] {$ \gamma_{f,k} \epsilon_k < 1$ and $\sum_{k=1}^\infty
	\gamma_{f,k}\epsilon_k =\infty$};
\item[(c)] $ \lim_{k\to
	\infty}\frac{|\epsilon_{k-1}-\epsilon_k|}{\gamma_{f,k}\epsilon^2_k}=0$;
\item[(d)] $\gamma_{g,k} \triangleq \gamma_g$ such that {$\gamma_g< 2/G_g$}, {$\lim_{k\to
	\infty}\frac{q_g^{k}}{\epsilon_k}=0$} and $\lim_{k\to
	\infty}\frac{q_g^{k-1}}{\gamma_{f,k}\epsilon^2_k}=0$, where {$q_g\triangleq \sqrt{1-\gamma_{g}\eta_g(2-\gamma_{g}{G_g})}$}.
\ee
\end{assumption}

\begin{lemma} \label{Lemma:tikh_1}
Let Assumptions ~\ref{assump:F lips cont} and ~\ref{assump:tikh stepsize} hold. Consider the sequences $\{x_k^t\}$ and $\{z_k^t\}$  generated by $(\mathrm{Tik}(\theta^*))$ and $(\mathrm{Tik}(\theta_k))$. Then, $\|x_k^t-z^t_k\| \to 0$ as $k\to \infty$.
\end{lemma}
\begin{proof}
By the definition of $x_k^t$, we have the following:
$ (z_k^t-x_k^t)^T(F(x_k^t;\theta_k)+\epsilon_k x_k^t) \geq 0. $
Similarly, we have the following:
$ (x_k^t-z_k^t)^T(F(z_k^t;\theta^*)+\epsilon_k z_k^t) \geq 0. $
By adding the two inequalities, we obtain the following:
$$ (z_k^t-x_k^t)^T(F(x_k^t;\theta_k)-F(z_k^t;\theta^*)) \geq
\epsilon_k\|x_k^t-z_k^t\|^2. $$
By using monotonicity of $F(\bullet;\theta^*)$ and Lipschitz continuity of $F$, this inequality can be recast as follows:
\begin{align*}
\epsilon_k\|x_k^t-z_k^t\|^2 & \leq
(z_k^t-x_k^t)^T(F(x_k^t;\theta_k)-F(z_k^t;\theta^*)) \\
					& =
					(z_k^t-x_k^t)^T(F(x_k^t;\theta_k)-F(x_k^t;\theta^*)+F(x_k^t;\theta^*)-F(z_k^t;\theta^*)) \\
			& =
			(z_k^t-x_k^t)^T(F(x_k^t;\theta_k)-F(x_k^t;\theta^*))
			 +\underbrace{(z_k^t-x_k^t)^T(F(x_k^t;\theta^*)-F(z_k^t;\theta^*))}_{\leq
				 0}\\
			 & \leq
			(z_k^t-x_k^t)^T(F(x_k^t;\theta_k)-F(x_k^t;\theta^*))
			 \leq \|z_k^t-x_k^t\| L_{F,\theta} \|\theta_k-\theta^*\|.
\end{align*}
It can then be concluded that
$$ \|z_k^t -x_k^t\| \leq
\frac{L_{F,\theta}}{\epsilon_k}\|\theta_k-\theta^*\| \leq \frac{L_{F,\theta}}{\epsilon_k}q_g^k\|\theta_0-\theta^*\|,$$
where the second inequality is a consequence of the strong convexity of
$g$, by which  $\theta_k$ converges
to $\theta^*$ at a geometric rate {$q_g\triangleq \sqrt{1-\gamma_{g}G_g(2-\gamma_{g}{G_g})}$}. Using Assumption~\ref{assump:tikh stepsize}(d), it follows that
$ \lim_{k \to \infty} \|z_k^t-x_k^t\| = 0. $
   \end{proof}

{We now develop a bound on $\|x_k^t - x_{k-1}^t\|$ in terms of the
	regularization parameters $\epsilon_k$ and $\epsilon_{k-1}$ and the
		estimates $\theta_k$ and $\theta_{k-1}$.
\begin{lemma}\label{Lemma:tikh_2}
	{Let Assumptions ~\ref{assump:F lips cont}, ~\ref{assump:VI_X_compact} and ~\ref{assump:tikh stepsize}(d) hold}. Suppose $x_k^t$ and $x_{k-1}^t$ are defined by Tik$(\theta_k)$ and
	Tik$(\theta_{k-1})$ respectively. Then, we have that
	$\|x_k^t-x_{k-1}^t\|$ can be bounded as follows:
$$\|x_k^t - x_{k-1}^t\|  \leq
\frac{L_{{F},\theta}q_g^{k-1}C_g}{\epsilon_k} + \frac{M}{\epsilon_k}
|\epsilon_{k-1}-\epsilon_{k}|,$$
{where {$q_g\triangleq \sqrt{1-\gamma_{g}G_g(2-\gamma_{g}{G_g})}$},
	$C_g\triangleq \|\theta_0-\theta^*\|(1+q_g) $, and $M$ is defined in
		Assumption~\ref{assump:VI_X_compact}. }
\end{lemma}
\begin{proof}
We begin by recalling that $x_{k-1}^t$ and $x_k^t$ satisfy the following
inequalities:
\begin{align*}
(x_{k-1}^t - x_k^t)^T (F(x_k^t,\theta_k) + \epsilon_k x_k^t) \geq 0,
	\mbox{ and }
(x_{k}^t - x_{k-1}^t)^T (F(x_{k-1}^t,\theta_{k-1}) + \epsilon_{k-1}
		x_{k-1}^t) \geq 0.
\end{align*}
Adding both inequalities, we obtain that
\begin{align*}
	(x_{k-1}^t - x_k^t)^T (F(x_k^t,\theta_k) - F(x_{k-1}^t,\theta_{k-1}))  +
	(x_{k-1}^t - x_k^t)^T (\epsilon_k x_k^t - \epsilon_{k-1}
		x_{k-1}^t) \geq 0.
\end{align*}
By adding and subtracting $(x_{k-1}^t - x_k^t)^TF(x_{k-1}^t,\theta_k)$ and $(x_{k-1}^t - x_k^t)^T\epsilon_{k}
x_{k-1}^t$, we obtain the following by using the monotonicity of
$F(x,\theta)$ in $x$:
\begin{align*}
& \quad 	(x_{k-1}^t - x_k^t)^T (F(x_{k-1}^t,\theta_k) -
			F(x_{k-1}^t,\theta_{k-1}))  +
	(x_{k-1}^t - x_k^t)^T (\epsilon_k x_{k-1}^t - \epsilon_{k-1}
		x_{k-1}^t)  \\
& \geq (x_{k-1}^t - x_k^t)^T (F(x_{k-1}^t,\theta_k) -
			F(x_{k}^t,\theta_{k})) +
	\epsilon_{k} (x_{k-1}^t - x_k^t)^T (x_{k-1}^t - x_{k}^t)  \geq \epsilon_k \|x_{k-1}^t - x_k^t\|^2.
\end{align*}
Consequently, by leveraging Cauchy-Schwartz inequality and by invoking
the bound $\|x\| \leq M$, we obtain the following bound:
\begin{align*}
\epsilon_k \|x_{k-1}^t - x_k^t\|^2 & \leq L_{F,\theta} \|x_{k-1}^t - x_k^t \|\|\theta_k
- \theta_{k-1}\| + \|x_{k-1}^t\|\|x_{k-1}^t -x_k^t\|
|\epsilon_{k-1}-\epsilon_{k}| \\
\implies  \|x_{k-1}^t - x_k^t\| & \leq
\frac{L_{F,\theta}}{\epsilon_k}\|\theta_k
- \theta_{k-1}\| + \frac{\|x_{k-1}^t\|}{\epsilon_k}
|\epsilon_{k-1}-\epsilon_{k}|  \leq
\frac{L_{F,\theta}}{\epsilon_k}\|\theta_k
- \theta_{k-1}\| + \frac{M}{\epsilon_k}
|\epsilon_{k-1}-\epsilon_{k}|.
\end{align*}
Furthermore, $\|\theta_k-\theta_{k-1}\|$ can be bounded as follows:
\begin{align*}
\|\theta_k-\theta_{k-1}\| &\leq \|\theta_k-\theta^*\|+\|\theta_{k-1}-\theta^*\|
\leq q_g^k\|\theta^*-\theta_0\|+q_g^{k-1}\|\theta^*-\theta_0\| =
q_g^{k-1}C_g.
\end{align*}
The resulting bound on $\|x_k^t - x_{k-1}^t\|$ can be further simplified
as
$\|x_k^t - x_{k-1}^t\|  \leq
\frac{L_{F,\theta}q_g^{k-1}C_g}{\epsilon_k} + \frac{M}{\epsilon_k}
|\epsilon_{k-1}-\epsilon_{k}|.$
\end{proof}

Next, we proceed to derive a bound on the difference $\|x_{k+1}-x_k^t\|$.
\begin{lemma}
\label{Lemma:tikh_3}
{Let Assumptions ~\ref{assump:F lips cont}, ~\ref{assump:VI_X_compact} and ~\ref{assump:tikh stepsize}(d) hold.} Suppose $\{x_k\}$ and $\{x_k^t\}$ are sequences generated by
Algorithm ~\ref{alg3} and $(\mathrm{Tik}(\theta_k))$. Then,  $\|x_{k+1}-x^t_k\|$ can be bounded as follows:
\begin{align*}
\|x_{k+1}-x^t_k \|\leq q_k\|x_{k}-x^t_{k-1}\|+
\frac{q_kL_{f,\theta}q_g^{k-1}C_g}{\epsilon_k} + \frac{Mq_k}{\epsilon_k}
|\epsilon_{k-1}-\epsilon_{k}|,
\end{align*}
where  $q_k\triangleq
\sqrt{(1+\gamma^2_{f,k}(L_{F,x}+\epsilon_k)^2-2\gamma_{f,k}\epsilon_k)}$
{and $C_g$, $q_g$ and $M$ are constants defined in Lemma ~\ref{Lemma:tikh_2}.}
\end{lemma}
\begin{proof}
We begin by bounding $\|x_{k+1}-x^t_k\|$ by leveraging
the nonexpansivity of the Euclidean projector.
\begin{align*}
\|x_{k+1}-x^t_{k}\|^2&=\|\Pi_X(x_k- \gamma_{f,k}(F(x_k;\theta_k)+\epsilon_k x_k)) - \Pi_X(x^t_k- \gamma_{f,k}(F(x^t_k;\theta_k)+\epsilon_k x^t_k))\|^2\\
&\leq \|x_k- \gamma_{f,k}(F(x_k;\theta_k)+\epsilon_k x_k) - (x^t_k- \gamma_{f,k}(F(x^t_k;\theta_k)+\epsilon_k x^t_k)\|^2\\
 & =  \|x_k-x^t_k\|^2+\gamma^2_{f,k}\|F(x_k,\theta_k)+\epsilon_k x_k-(F(x^t_k;\theta_k)+\epsilon_k
	x^t_k)\|^2 \notag\\&-2\gamma_{f,k}(x_k-x^t_k)^T(F(x_k;\theta_k)+\epsilon_k x_k-(F(x^t_k;\theta_k)+\epsilon_k x^t_k).
\end{align*}
The Lipschitzian property of $F(x;\theta)$ in
$x$ uniformly in $\theta$ and the strong monotonicity of
$(F(x;\theta)+\epsilon x)$ in $x$ uniformly in $\theta$ allows for deriving the following bound.
\begin{align*}
& \quad \|x_k-x^t_k\|^2+\gamma^2_{f,k}\|F(x_k,\theta_k)+\epsilon_k x_k-(F(x^t_k,\theta_k)+\epsilon_k
	x^t_k)\|^2 \notag\\&-2\gamma_{f,k}(x_k-x^t_k)^T(F(x_k,\theta_k)+\epsilon_k x_k-(F(x^t_k,\theta_k)+\epsilon_k x^t_k) \\
	&\leq \|x_k-x^t_k\|^2+\gamma^2_{f,k}(L_{F,x}+\epsilon_k)^2\|x_k-x^t_k\|^2-2\gamma_{f,k}\epsilon_k\|x_k-x^t_k\|^2\\
&=(1+\gamma^2_{f,k}(L_{F,x}+\epsilon_k)^2-2\gamma_{f,k}\epsilon_k)\|x_k-x^t_k\|^2,
\end{align*}
which can be simplified to
$\|x_{k+1}-x^t_{k}\| \leq q_k  \|x_k-x^t_k\|$
where $q_k\triangleq
(1+\gamma^2_{f,k}(L_{F,x}+\epsilon_k)^2-2\gamma_{f,k}\epsilon_k)^{1/2}$. By
using the triangle inequality, the above inequality can be expanded as the following:
\begin{align}
\|x_{k+1}-x^t_{k}\| &\leq q_k  \|x_k-x^t_k\|
\leq q_k  \|x_k-x^t_{k-1}\| +q_k\|x^t_k-x^t_{k-1}\|\label{tikh_lemma1_ineq1}
\end{align}
By combining ~\eqref{tikh_lemma1_ineq1} and Lemma~\ref{Lemma:tikh_2}, we obtain the following:
\begin{align*}
\|x_{k+1}-x^t_k \|\leq q_k\|x_{k}-x^t_{k-1}\|+
\frac{q_kL_{f,\theta}q_g^{k-1}C_g}{\epsilon_k} + \frac{Mq_k}{\epsilon_k}
|\epsilon_{k-1}-\epsilon_{k}|.
\end{align*}
\end{proof}
\noindent We now leverage this bound to show that $\|x-x_k^t\| \to 0$ as $k \to \infty$.
\begin{lemma}
\label{Lemma:tikh_4}
{Let Assumptions ~\ref{assump:F lips cont}, ~\ref{assump:VI_X_compact} and ~\ref{assump:tikh stepsize} hold.} Consider the sequence $\{x_k\}$ and $\{x_k^t\}$ generated by
Algorithm ~\ref{alg3} and
	$(\mathrm{Tik}(\theta_k))$, respectively. Then, $\lim_{k\to \infty} \|x_k -x_k^t\| = 0.$
\end{lemma}
\begin{proof}
This requires the use of Lemma ~\ref{Lemma:tikh_3} and Lemma
	~\ref{polyak}.\\
\noindent {(i)} Under Assumption ~\ref{assump:tikh stepsize}(a), we have that
\begin{align*}
q_k  & = \sqrt{(1+\gamma^2_{f,k}(L_{F,x}+\epsilon_k)^2-2\gamma_{f,k}\epsilon_k)}
	 = \sqrt{(1-\gamma_{f,k} (2\epsilon_k  -
				\gamma_{f,k}(L_{F,x}+\epsilon_k)^2))}
	 \leq \sqrt{(1-\gamma_{f,k}\epsilon_k)}{<1,} 	
\end{align*}
{where the last inequality follows from Assumption ~\ref{assump:tikh stepsize}(b).}
Hence, we obtain the following:
\begin{align*}
\sum_{k=1}^\infty (1-q_k)  & = \sum_{k=1}^\infty \frac{(1-q^2_k)}{1+q_k}
		 \geq {1\over 2}\sum_{k=1}^\infty
						{(1-q^2_k)}{ \ \geq \ } {1\over 2}\sum_{k=1}^\infty \gamma_{f,k}\epsilon_k  =\infty,
\end{align*}
{where the last equality follows from Assumption~\ref{assump:tikh stepsize}(b).}\\
\noindent {(ii)} Under Assumption~\ref{assump:tikh stepsize}, we
obtain the following:
\begin{align*}
 & \quad \lim_{k\to\infty}
\left[\frac{q_kL_{f,\theta}q_g^{k-1}C_g}{(1-q_k)\epsilon_k} +
\frac{Mq_k}{(1-q_k)\epsilon_k}
|\epsilon_{k-1}-\epsilon_{k}| \right]\\
&= \lim_{k\to\infty}\left[ \frac{(1+q_k)q_kL_{f,\theta}q_g^{k-1}C_g}{(1-q^2_k)\epsilon_k} +
\frac{Mq_k(1+q_k)}{(1-q^2_k)\epsilon_k}
|\epsilon_{k-1}-\epsilon_{k}|  \right]\\
& \leq  \lim_{k\to\infty}\left[
\frac{(1+q_k)q_kL_{f,\theta}q_g^{k-1}C_g}{\gamma_{f,k}\epsilon^2_k} +
\frac{Mq_k(1+q_k)}{\gamma_{f,k}\epsilon^2_k}
|\epsilon_{k-1}-\epsilon_{k}| \right] \\
& \leq  \lim_{k\to\infty}\left[
\frac{2L_{f,\theta}q_g^{k-1}C_g}{\gamma_{f,k}\epsilon^2_k} +
\frac{2M}{\gamma_{f,k}\epsilon^2_k}
|\epsilon_{k-1}-\epsilon_{k}| \right]= 0,
\end{align*}
where the last inequality  follows from {Assumption ~\ref{assump:tikh
	stepsize}(b) (since
	$\gamma_{f,k} \epsilon_k < 1$ for all $k$ implying $q_k < 1$)}  and
	{the last equality is a consequence of invoking } Assumption
\ref{assump:tikh stepsize} (d) and (c).  Hence, conditions of Lemma~\ref{polyak} are met. This completes the proof.
\end{proof}

We now prove the convergence of the regularized gradient schemes by
	showing that $\|x_k^t - z_k^t\| \to 0$ as $k \to \infty$. 
\begin{theorem}[{\bf Convergence of regularized scheme}]
{Let Assumptions ~\ref{assump:F lips cont}, ~\ref{assump:VI_X_compact} and ~\ref{assump:tikh stepsize} hold.} Consider the
	sequence $\{x_k,\theta_k\}$  generated by Algorithm ~\ref{alg3}.
	Then, $\{x_k\}$ converges to $x^*$ as {$k\to\infty$}, where $x^*$ denotes the
	least-norm solution of $X^*$ and
	$\{\theta_k\}$ converges to $\theta^*\in \Theta$.
\end{theorem}
\begin{proof}
From Lemma~\ref{Lemma:tikh_4}, it can be concluded that $x_k \to x^t_k$ as
$k\to\infty$. Furthermore, Lemma ~\ref{Lemma:tikh_1} guarantees that
$x^t_k \to z^t_k$ as $k\to\infty$. Moreover, the sequence of solutions
to the  (Tikhonov) regularized problems, denoted by $\{z_k^t\}$,
   converges to $x^*$, the least norm solution of VI$(X,
		   F(\bullet;\theta^*))$ (cf.~\cite[Ch.~12]{facchinei02finite}).
  It follows that $x_k\to x^*$ as $k\to \infty$.
\end{proof}

A natural question is whether there is indeed a feasible choice of
steplength sequences that satisfies the prescribed assumptions. In the
next Lemma, we show that there exists a feasible choice of stepsizes
that can satisfy requirements of Assumption ~\ref{assump:tikh stepsize}.
\begin{lemma} Let $\gamma_{f,k}=\frac{1}{(L_{F,x}+1)^2(k+1)^\alpha}$ and
$\epsilon_k=\frac{1}{(k+1)^\beta}$, where $0<\beta<\alpha<1$ and {$0<\alpha+\beta<1$}. Then, conditions of Assumption ~\ref{assump:tikh stepsize} are satisfied.
\end{lemma}
\begin{proof}
\noindent (a) {It can be seen by the choices of $\gamma_{f,k}$ and
	$\epsilon_k$ that
	$$ \gamma_{f,k} = \frac{1}{(L_{F,x}+1)^2 (k+1)^{\alpha}} \leq \frac{1}{(L_{F,x}+1)^2
		(k+1)^{\beta}} \leq
		\overbrace{\frac{1}{(L_{F,x}+\frac{1}{k+1}^{2\beta})(k+1)^{\beta}}}^{\triangleq
			\epsilon_k/(L_{F,x}+\epsilon_k^2)} \leq
			\frac{1}{L_{F,x}^2(k+1)^{\beta}}.$$
}
\noindent (b)
$\sum_{k=1}^\infty\gamma_{f,k}\epsilon_k=\sum_{k=1}^\infty
\frac{1}{(L_{F,x}+1)^2(k+1)^{\alpha+\beta}} \geq \sum_{k=1}^\infty
\frac{1}{k+1}=\infty.$\\
\noindent (c) If $t \triangleq (k+1)$, we may express the required limitas follows:
{\begin{align*}
&\lim_{k\to\infty}\frac{\epsilon_{k-1}-\epsilon_{k}}{\gamma_{f,k}\epsilon^2_{k}}=\lim_{k\to\infty} \frac{ {1\over k^\beta}-{1\over (k+1)^\beta}} {{1\over (k+1)^{\alpha+2\beta}}}=
\lim_{k\to\infty} \left({k+1\over k}\right)^\beta
\lim_{k\to\infty}\frac { (k+1)^\beta -
	k^\beta}{(k+1)^{-\alpha}}=\lim_{k\to\infty}\frac{1-\left({k\over
			k+1}\right)^\beta}{(k+1)^{-\alpha-\beta}} \\
& = \lim_{t\to\infty}\frac{1-\left(1 -
		\frac{1}{t}\right)^\beta}{t^{-\alpha-\beta}}.
\end{align*}}
Since this limit is of the form of $0/0$, we may use L'H\^{o}pital's
rule to express  the limit as follows:
{\begin{align*}
\lim_{t\to\infty}\frac{1-\left(1 -
		\frac{1}{t}\right)^\beta}{t^{-\alpha-\beta}} & =
\lim_{t\to\infty} \frac{ -\beta \left(1 -
		\frac{1}{t}\right)^{\beta-1}
\frac{1}{t^2}}{(-\alpha-\beta)t^{-\alpha-\beta-1}}  =
\lim_{t\to\infty} { -\beta \left(1 -
		\frac{1}{t}\right)^{\beta-1}
} \lim_{t\to\infty} \frac{1}{(-\alpha-\beta)t^{1-\alpha-\beta}} = 1
\times 0 = 0.
\end{align*}}
\noindent (d) {We have that
$\lim_{k\to\infty}\frac{q_g^{k-1}}{\gamma_{f,k}\epsilon^2_{k}}=\lim_{k\to\infty}
\frac{{q_g^{k-1}}}{{1\over (k+1)^{\alpha+2\beta}}}=0,$ since the
numerator converges to zero at a faster rate than the denominator.} In addition, {$\lim_{k\to
	\infty}\frac{q_g^{k}}{\epsilon_k}=\lim_{k\to\infty}
\frac{{q_g^{k}}}{{1\over (k+1)^{\beta}}}=0$ for the same reason.}
\end{proof}

\section{Numerical Results}\label{sec:Numerics}
In this section, we present some numerical results that support the
converence and rate analysis provided earlier. In Section~\ref{sec:41},
		   we describe the  economic dispatch problem which will form
		   the basis of our computational investigations. On the basis
		   of this problem, we consider the problem of misspecified
		   costs (Section~\ref{sec:42}) as well as misspecified demand
		   (Section~\ref{sec:43}).

\subsection{Economic dispatch problem}\label{sec:41}
 A traditional economic dispatch problem~\cite{kirschen2004a} requires scheduling of
generation to meet demand requirements in a least-cost fashion. The
schedule has to abide by a set of capacity and ramping constraints and
is given by  the following optimization problem:
 \begin{align}
 \min & \quad \sum_{t=1}^T\sum_{i=1}^N  c_i(g_{i,t}) \tag{EDisp}\\
 \st & \quad \sum_{i=1}^N g_{i,t}  \geq d_{t}, & \forall   t=1,\hdots,T\label{cosn:balance equ}\\
 & \quad 0   \leq g_{i,t}  \leq G_i   & \forall i,\,  t=1,\hdots,T\label{cons:max gen1}\\
 & \quad  g_{i,t}-g_{i,t-1} \leq {r}^{\rm up}_i \quad
   & \forall  i, \,  t=2,\hdots,T\label{cons:ramp rate1}\\
    & \quad   g_{i,t-1}-g_{i,t} \leq {r}^{\rm down}_i \quad
   & \forall  i, \,  t=2,\hdots,T,\label{cons:ramp rate2}
 \end{align}
{ where $N$ and $T$ are number of generators and time periods, respectively}. In addition, $g_{i,t}$ represents output power of generator $i$ at time $t$,
 and $c_i(\cdot)$ is the generation cost function of generator $i$,
 $d_t$ denotes load demand at period $t$, $G_i$ is the capacity of
 generator $i$, and ${r}^{\rm up}_i $ and ${r}^{\rm down}_i $ are the
 ramp-up and ramp-down limits of generator $i$, respectively. Note that
 ~\eqref{cosn:balance equ} is responsible for balancing generation with
 demand while ~\eqref{cons:max gen1} ensures that the power output of
 generators stay within the defined threshold. Constraints
 ~\eqref{cons:ramp rate1} and ~\eqref{cons:ramp rate2} are ramping rate
 bounds that simply ensure that any change in power output is
 within  a defined limit over consecutive periods.
\subsection{Misspecified cost functions}\label{sec:42}
In what follows, we consider a setting  where generation cost functions
are misspecified quadratic functions modeled as
$c_i(g;\theta_i)=\theta_{i1}g^2+ {\theta_{i2}} g$, where
${\theta_i}=(\theta_{i1},\theta_{i2})$ is unknown. Suppose that for
generator $i$, we have a prior collection of {$P$} samples denoted by
$(c_{ij},g_{ij})$, $j=1,\hdots,{P}$ defined as
$c_{ij}={\theta^*_{i1}}g_{ij}^2+ {\theta^*_{i2}} g_{ij}+\xi_j(\omega)\quad j=1,\hdots,P,$
 where $\xi$ is a random variable with mean zero. Then, the misspecified parameter {$\theta^*=(\theta^*_{i1}, \theta^*_{i2})_{i=1}^{i=N}$} is learnt by solving the
following least squares problem:
$$ \min_{{\theta\in \R^{2}\times\R^N}} h(\theta), \mbox{ where } h(\theta)
	\triangleq {1 \over
{NP}}\sum_{i=1}^N{\sum_{j=1}^P}(c_{ij}-(\theta_{i1}g_{ij}^2+ \theta_{i2} g_{ij}))^2.$$
\begin{table}[htbp]
\scriptsize
\begin{center}
\begin{tabular}{|c|c|c|c|c|c|c|c|c|c|c|}
  \hline
  $\#$ & Capacity    & $r^{\rm up}$ &  $r^{\rm down}$   \\
  \hline  \hline
  $1$                 & $40$    &$20$     &$20$          \\
  \hline
  $2$                 &$40$    &$20$     &$20$     \\
  \hline
  $3$                 & $35$   & $18$            & $18$       \\
  \hline
  $4$               & $50$     & $25$       & $25$         \\
  \hline
  $5$                   & $40$    &$20$      & $20$      \\
  \hline
\end{tabular}
\end{center}
\caption{Generator capacities and ramp limits}
\label{tab1}
\end{table}
In the first set of tests, we examine convergence of the constant and diminishing step
length schemes proposed in Section ~\ref{subsec:Str Opt Str learn}. We consider a set of
$5$ generators with misspecified generation cost function coefficients.
The goal is to schedule the power output over $5$ time periods. The
generators' specifications are shown in Table ~\ref{tab1}. For each
generator, a set of $1000$ samples is collected for constructing the
learning problem.
\begin{figure}[ht]
\centering
  \centering
 \includegraphics[scale=0.48]{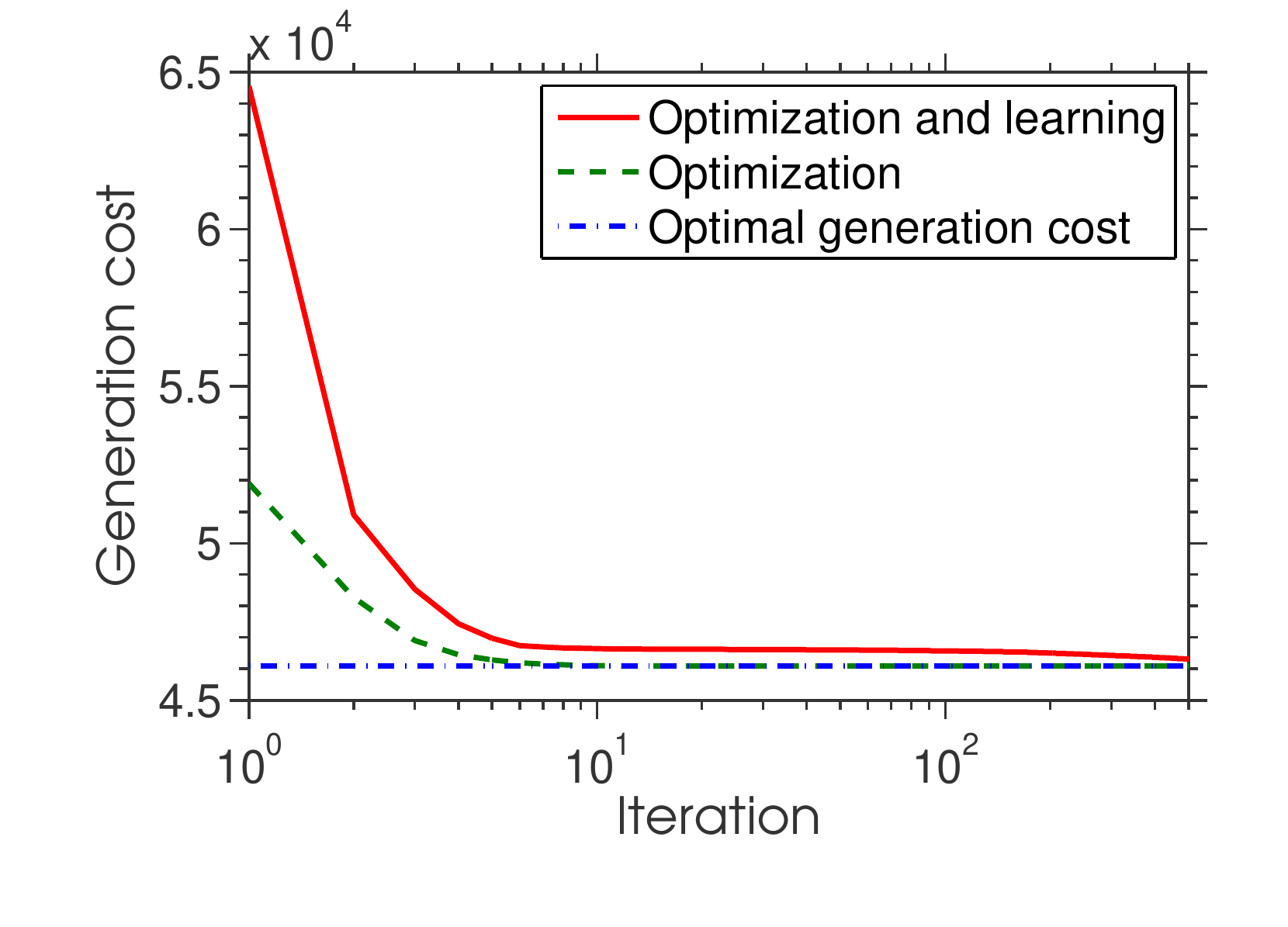}
  \centering
\includegraphics[scale=0.48]{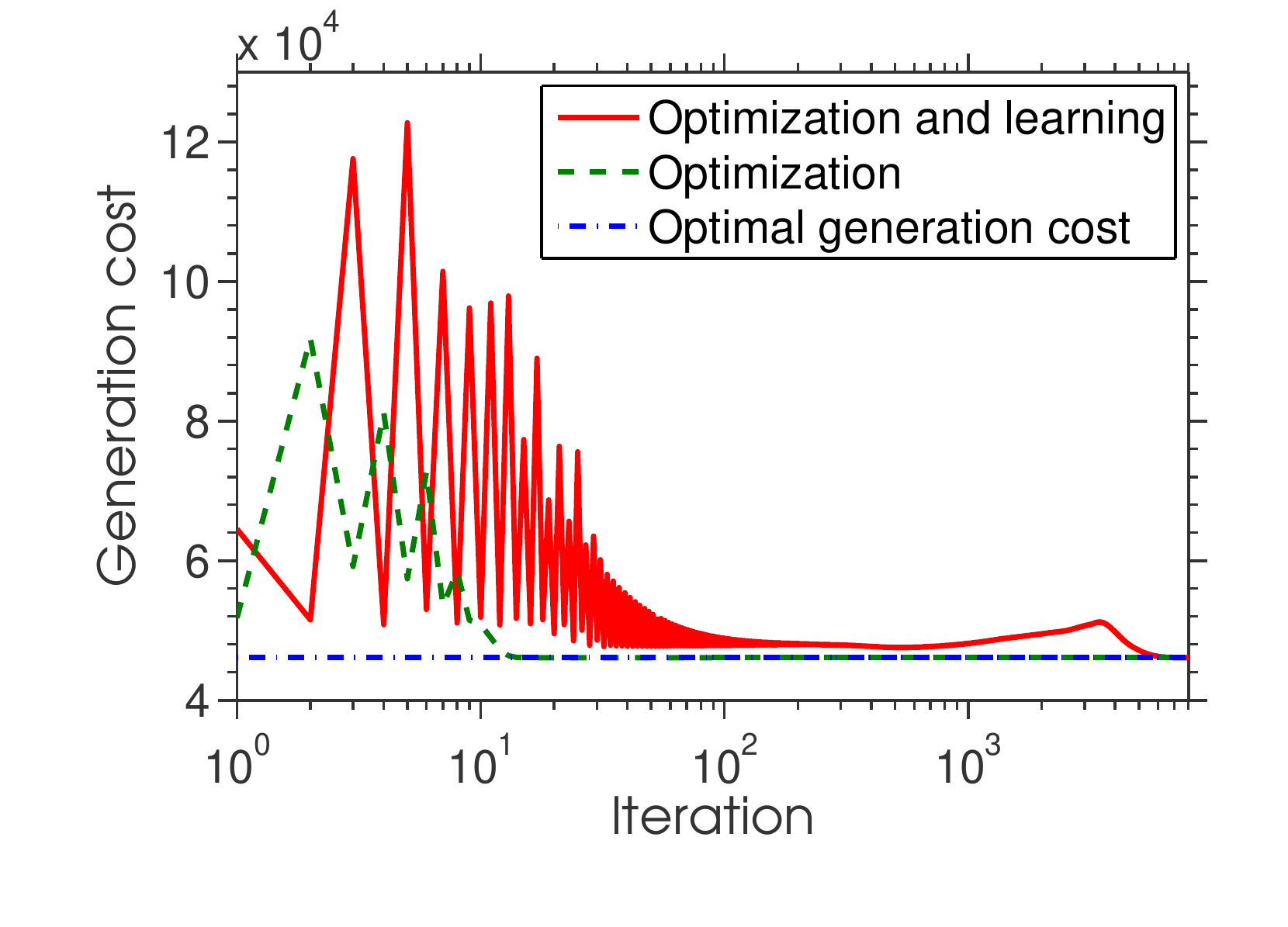}
\caption{{Strongly convex Opt. and
	learning:  Const. steplength (l) and Diminish. steplength
		(r)}}
\label{Fig:Conv_strconv_const_diminsh}
\end{figure}
Figure~\ref{Fig:Conv_strconv_const_diminsh} shows the behavior when
using a constant steplength scheme, with $\gamma_f=0.04$ and
$\gamma_g=0.003$. {Note that the Lipschitz constants for the gradient of
optimization and learning functions are $G_{f,x}=20$ and $G_g=250$,
			 respectively, while the strong convexity constant of the optimization problem is $\eta_f=20$}. Hence,
the prescribed stepsizes satisfy the required conditions. The scheme is also compared to the case when using the
optimal $\theta^*$ in the cost coefficient, requiring no learning.
Expectedly, we observe slower convergence when the cost function
coefficients are misspecified. {The figure on the right} displays the
trajectories when using diminishing step length scheme with
$\gamma_{f,k}=\gamma_{g,k}={1\over k}$. Figure ~\ref{Fig: Rate analysis,
	str opt str learn} plots the
convergence rate when using constant step size schemes and both the
optimization and learning problems are strongly convex.

\begin{figure}[H]
\centering
  \centering
\includegraphics[scale=0.48]{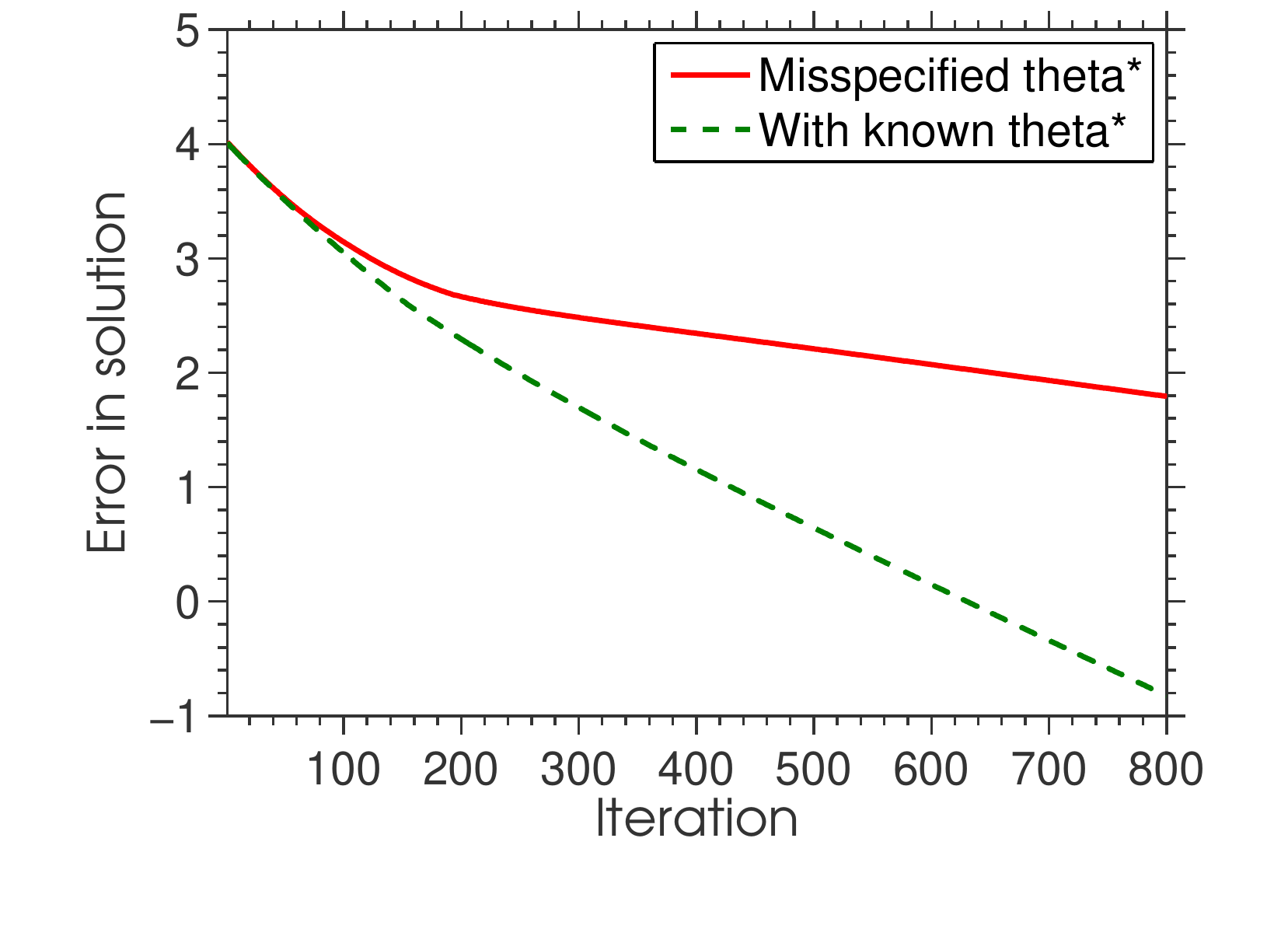}
  \centering
\includegraphics[scale=0.48]{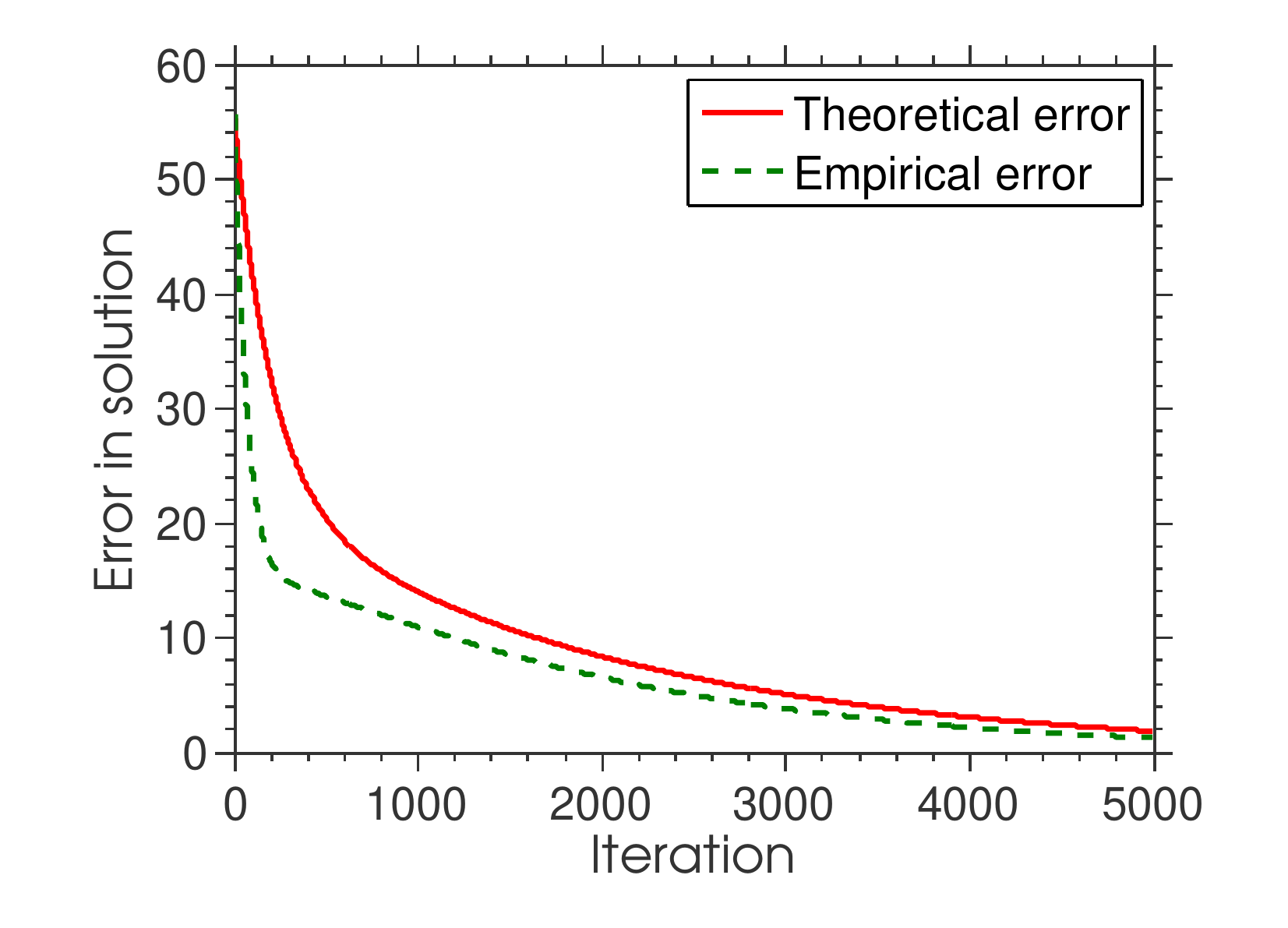}
\caption{{Strongly convex optimization: Impact on rate (l)
	and empirical vs. theor. rate (r)}}
\label{Fig: Rate analysis, str opt str learn}
\end{figure}
{Figure
~\ref{Fig: Rate analysis, str opt str learn} (l)} compares the error in
solution iterates of optimization problem for the cases of missipecified
and known $\theta^*$. As would be expected, when no learning is
involved, we observe a linear convergence rate as shown in  the dashed
line.  However, when learning is incorporated, the rate drops as shown
by solid line. {Figure ~\ref{Fig: Rate analysis, str opt str learn} (r)}
compares the actual error in solution iterates of misspecified
optimization problem to the theoretically predicted bound obtained in
Proposition ~\ref{Prop: con_rate_strconv} and supports the validity of
the bound.
\begin{table}[htbp]
\scriptsize
\begin{center}
\begin{tabular}{|c|c|c|c|c|c|c|}
  \hline
  \multirow{2}{*}{  No. of generators} &  \multicolumn{2}{c|} {Constant step size, $k=5000$}&
	  \multicolumn{2}{c|}{ Diminishing step size, $k=15000$}&  \multicolumn{2}{c|} {Averaging scheme,  $k=15000$} \\\cline{2-7}
   &$\|\theta_k-\theta^*\|$ &$\frac{\|f(g_k,\theta^*)-f^*\|}{1+f^*}$&$\|\theta_k-\theta^*\|$ &$\frac{\|f(g_k,\theta^*)-f^*\|}{1+f^*}$&$\|\theta_k-\theta^*\|$ &$\frac{\|f(g_k,\theta^*)-f^*\|}{1+f^*}$\\
  \hline  \hline
  5&$3.3$$e$-$3$& $9.4$$e$-$7$&$1.3$$e$-$3$&$5.7$$e$-$4$&$6.9$$e$-$7$&$3.4$$e$-$4$\\\hline
10&$1.2$$e$-$2$ &$2.7$$e$-$6$ &$2.7$$e$-$3$ &$6.0$$e$-$4$&$1.1$$e$-$6$ &$5.8$$e$-$4$ \\\hline
15&$1.2$$e$-$1$& $5.4$$e$-$5$&$1.2$$e$-$3$&$5.9$$e$-$4$&$2.1$$e$-$6$&$5.6$$e$-$4$\\\hline
20&$9.0$$e$-$1$&$3.0$$e$-$3$&$4.3$$e$-$2$&$1.2$$e$-$3$&$5.0$$e$-$6$&$6.6$$e$-$4$ \\\hline
\end{tabular}
\end{center}
\caption{Constant and  diminishing stepsize and averaging schemes}
        \label{table: Conv_Opt_Schemes}
\vspace{-0.1in}
\end{table}

{In Table~\ref{table: Conv_Opt_Schemes}, we examine the performance of
the various schemes as the
problem size grows. The implemented schemes are the constant step size scheme proposed in Proposition~\ref{prop_conv_csl} , diminishing step size scheme proposed in Proposition~\ref{Prop:dimin_step} and averaging scheme stated in proposition~\ref{theo_conv_convex}}. We compare the error in both the solution to the
learning problem and the error in the function value associated with the
optimization problem after a prescribed set of iterations. While
constant steplength schemes perform well, the performance appears to be
more affected by problem size in comparison with diminishing steplength
or averaging schemes. This can be traced to the observation that as
problem size grows, the Lipschitz constant of gradient of learn function
increases as well and the employed step sizes for constant step size
scheme are adjusted accordingly.

Figure ~\ref{Fig:Rate convex averaging} displays the performance  when
using the averaging schemes proposed in Proposition
~\ref{theo_conv_convex}. With known $\theta^*$, the rate of convergence
in function values is of the order of $1/K$ where $K$ is number of
steps.  In Figure ~\ref{Fig:Rate convex averaging} (l), the error in function values is shown as a dashed line when $\theta^*$ is known
and this rate drops by a constant factor when learning is involved as
shown by the solid line. Figure ~\ref{Fig:Rate convex averaging} (r)
compares the theoretical bound in Proposition ~\ref{theo_conv_convex}
with the empirical {error}. As
it is confirmed in this figure, the theoretically predicted rate represents an
upper bound to the actual convergence rate of averaging scheme. Table
~\ref{table: Conv_Opt_Schemes} displays the errors obtained from running
the averaging scheme for $15000$ iterations with increasing number of
generators.

\begin{figure}[H]
\centering
  \centering
\includegraphics[scale=0.48]{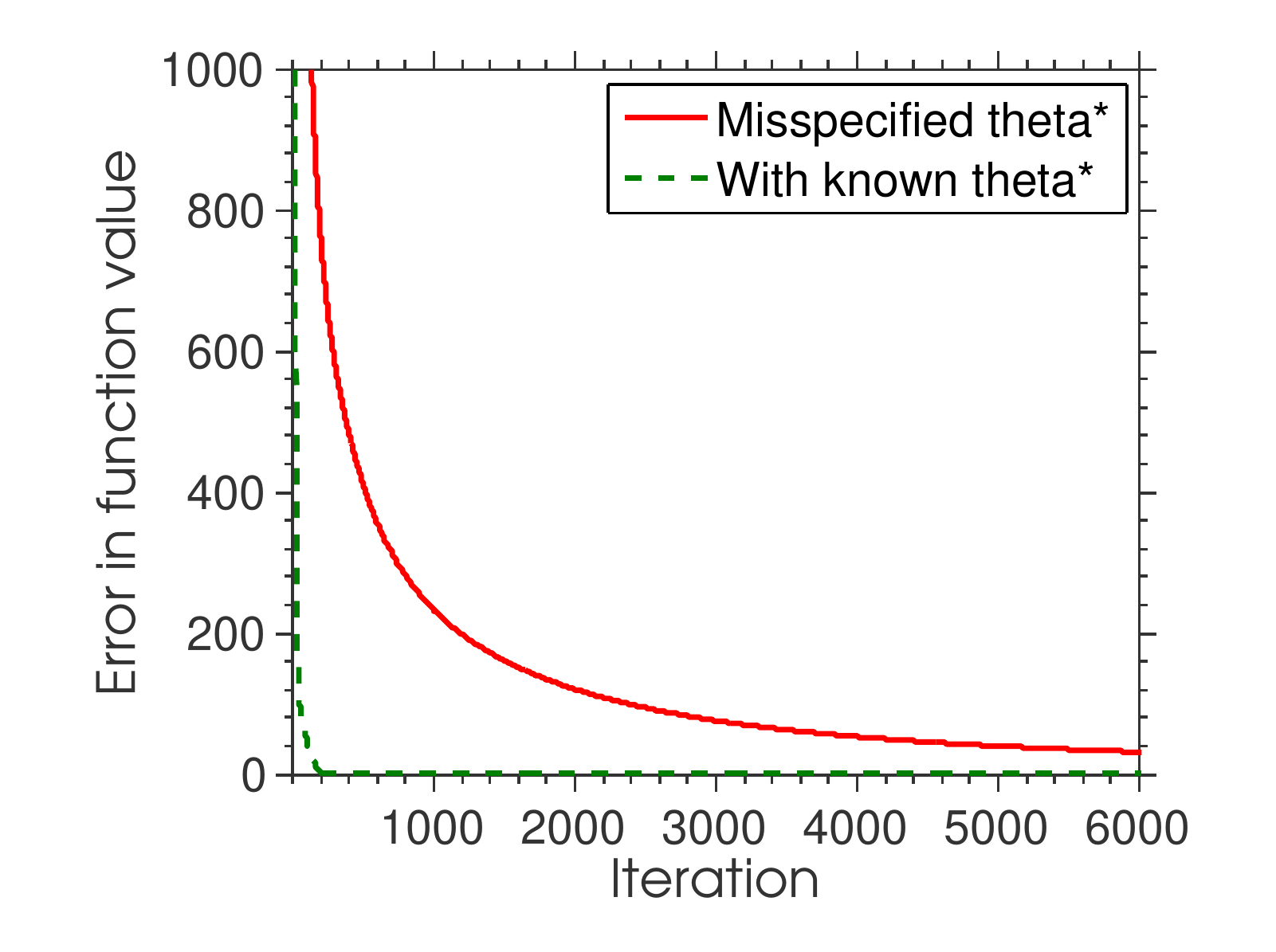}
  \centering
\includegraphics[scale=0.48]{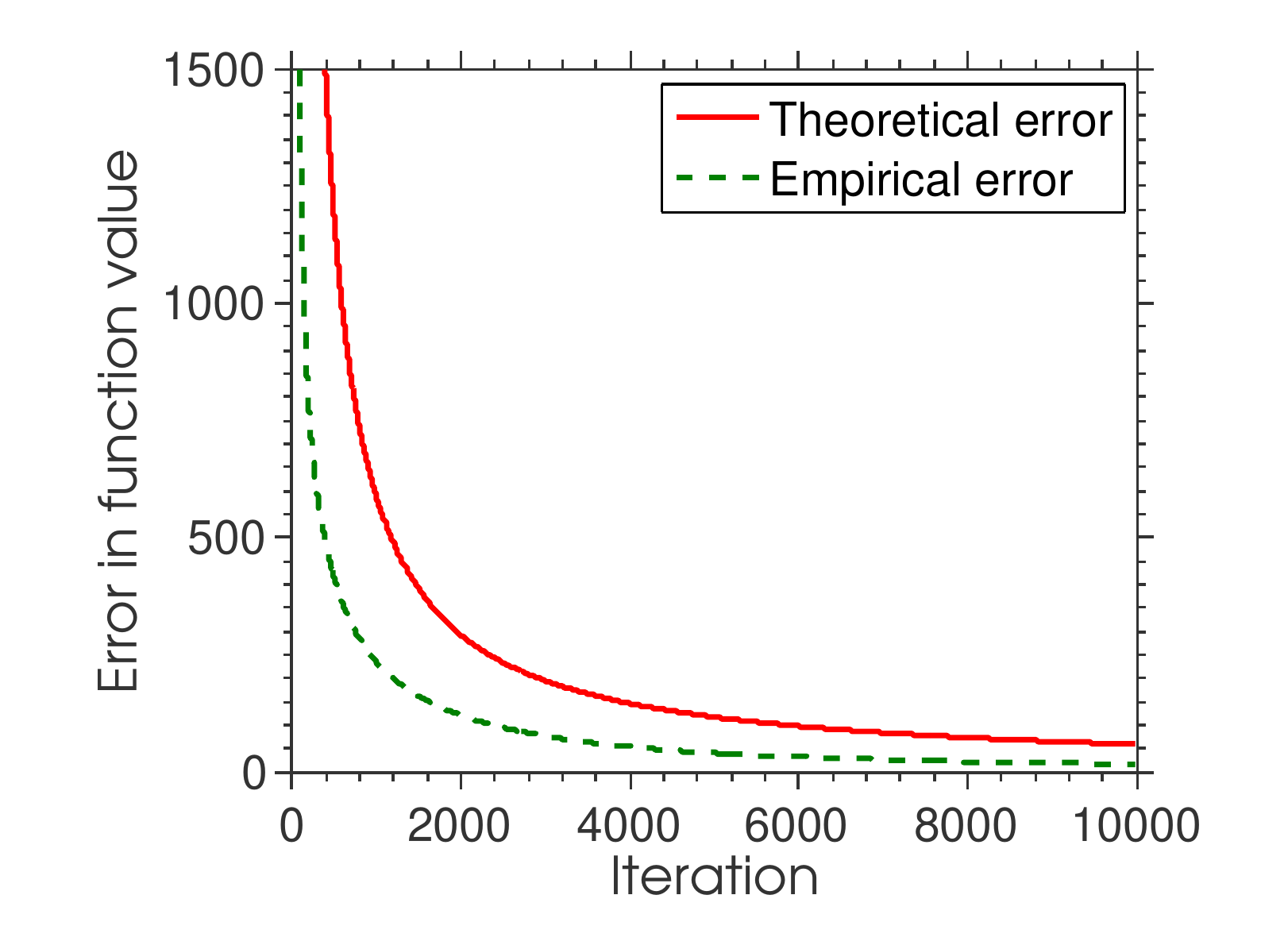}
\caption{Convex optimization: Impact on rate
	(l) and empirical vs. theor. (r)}
\label{Fig:Rate convex averaging}
\vspace{-0.2in}
\end{figure}

\begin{figure}[H]
\centering
\includegraphics[scale=0.54]{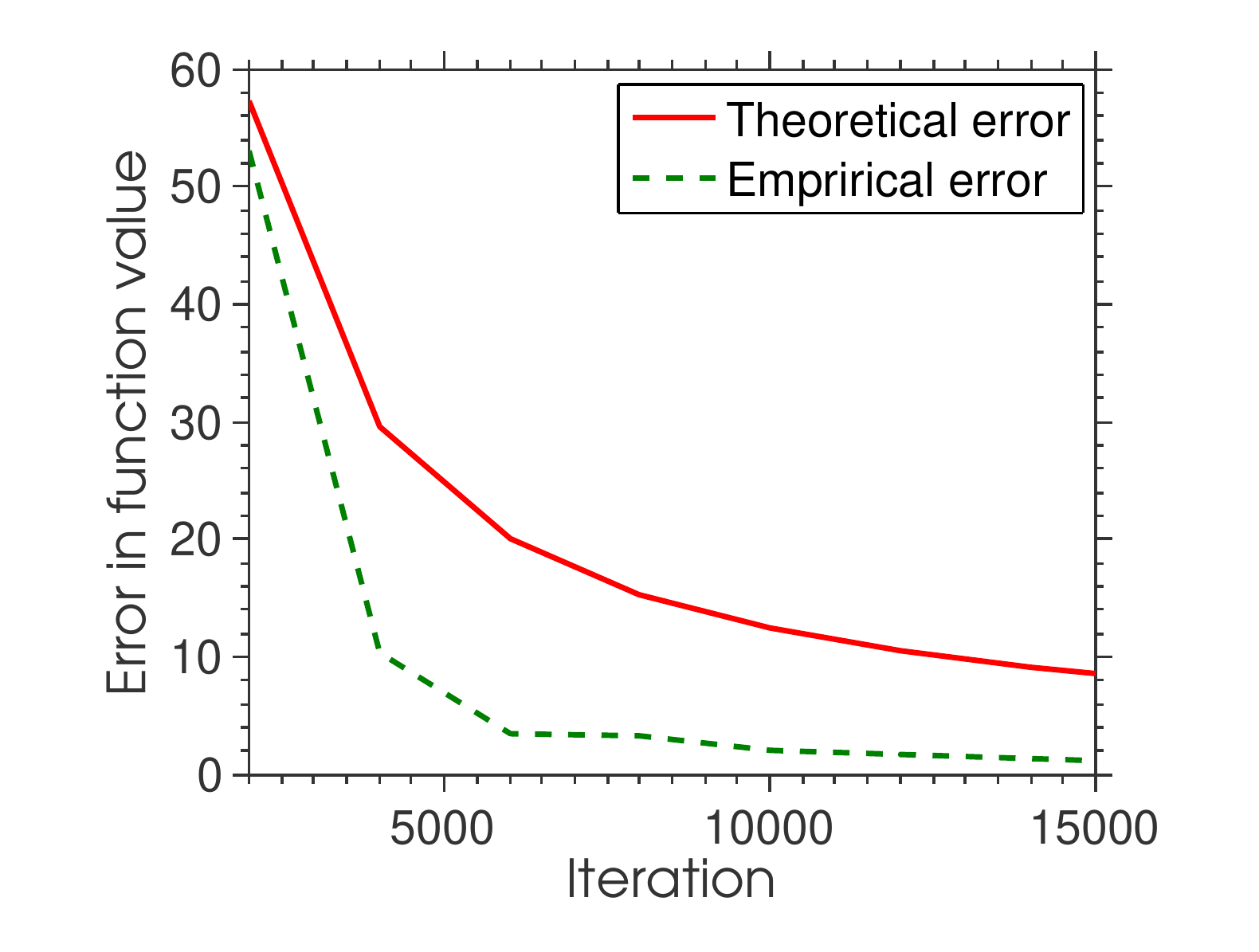}
\caption{Nonsmooth convex optimization: empirical error vs. theor. bound}
\label{Fig:Rate_sub_grad_av}
\end{figure}
{To test the joint subgradient scheme(Algorithm~\ref{alg2}), we consider a nonsmooth generation cost function that is the maximum of $3$ linear functions and is defined as below:
$$c_i(g;\theta_i)=\max\Big(\theta_{i1}g+\theta_{i2},\theta_{i3}g+\theta_{i4},\theta_{i5}g+\theta_{i6}\Big)\quad i=1,\dots,N$$
Figure ~\ref{Fig:Rate_sub_grad_av} displays the result using the optimal
constant step length scheme proposed in part (ii) of Proposition
~\ref{prop:sub_aver}. Given a terminal iteration index $K$, the optimal
step length is first calculated using ~\eqref{opt_step_rule} and then the scheme is terminated after $K$ number of iterations. Figure ~\ref{Fig:Rate_sub_grad_av} compares the resulted empirical error in function value of averaged point versus the theoretical bound. As shown in the figure, the empirical error is within the theoretical bound.}
\subsection{Misspecified demand}\label{sec:43}
Suppose that demand vector $d\triangleq(d_t:t=1,\hdots,T)$ is
misspecified and may be learnt through a parallel learning process.
We refer to the misspecifed problem as (EDisp$(d)$) where $d$ denotes
the misspecified demand. Suppose the linear inequality constraints of
(EDisp$(d)$) are given by
\begin{align*}
   h(g) \triangleq \pmat{
   \pmat{\sum_{i=1}^N g_{i,t} -d^*_t}_{t=1}^T\\
    \pmat{G_i -g_{i,t}}^T_{i,t=1}\\
	\pmat{{r}^{\rm up}_i  -g_{i,t}+g_{i,t-1}}_{i,t=2}^T\\
   \pmat{{r}^{\rm down}_i -g_{i,t-1}+g_{i,t}}_{i,t=2}^T },
 \end{align*}
where $g\triangleq(g_{i,t}:i=1,\hdots,N$, $t=1,\hdots,T)$,
	  {and the cost function is given by $c(g)\triangleq\sum_{t=1}^T\sum_{i=1}^N  c_i(g_{i,t}) $.} The first
	  order conditions of this problem are necessary and sufficient and
	  are given by
\begin{align*}
0\leq z \perp F(z;d^*)\geq 0 ,\label{formulation: complementary} \mbox{
	where }
  z\triangleq \pmat{g \\ \lambda},   F(z) \triangleq  \pmat{\nabla_g c(g;d^*)-\nabla_g
	   h(g;d^*)^T\lambda \\
		   h(g;d^*)},
 \end{align*}
and $\lambda$ is a vector of dual variables corresponds to the
constraints set $h(g) \geq 0$. The above conditions can be compactly
stated as VI$(Z,F(\bullet;d^*))$~\cite{Pang03I} allowing us to consider
the use of the regularized and extragradient schemes developed in
Section ~\ref{sec:monotone VI} for the solution of misspecified
variational inequality problems. We consider a set of $5$ generators
with known quadratic cost functions  while the demand vector
$d^*=(d^*_t:t=1,\hdots,T)$ is unknown. A set of $1000$ samples is randomly
generated and the optimal demand is the solution to the following
learning problem:
 \begin{align*}
 \min_{d\in \R_{+}^T} L(d)\quad \mbox{where} \quad L(d)\triangleq \sum_{i=1}^{1000} \|d-y_i\|^2,
\end{align*}
 and $y_i$, $i=1,\hdots,1000$ denote the set of samples. Since the
 variational problem is merely monotone, the solution set is
 multi-valued. In such settings, we use the gap
 function~\cite{facchinei02finite} as a metric of progress, {which is analogous to the objective function in optimization}. {Given VI$(Z,F)$, associated gap function is defined as} follows:
      $$ G(y) \triangleq \begin{cases} F(y)^Ty, & F(y) \in {Z}^{\circ} \\
		  							  +\infty, & F(y) \not \in Z^{\circ},
						\end{cases} $$
where ${Z}^{\circ} \triangleq \left\{z: z^Ty \geq 0, y \in Z\right\}.$
 Recall that $x$ solves VI$(Z,F)$ if and only if $G(x) = 0$. To allow for representing the gap function when $F(y) \not \in
{Z}^{\circ}$, we use a modified gap function given by $G(y) = F(y)^Ty$,
	which could be negative. 
Figure ~\ref{Fig:Con_VI} (l) compares the trajectory of gap function
value with learning (solid line) with the trajectory observed when $d^*$
	is available. Note that in this problem, $T=2$ and  $\gamma_f=
	{k^{-0.65}}$ and $\epsilon_k={k^{-0.34}}$. In addition, we employ a
constant step size of $\gamma_g=.003$ for the learning problem, given
	that the Lipschitz constant of $\nabla_d L(d)$ is estimated to be
	$520$. Expectedly, learning leads to a degradation in the
	convergence rate as compared with using  the true demand $d^*$.  In
	Figure ~\ref{Fig:Con_VI} (r), we examine the behavior of the
	misspecified extragradient scheme where	$T=5$ and $L_{F,x}=2.8$, $L_{F,\theta}=1$
	and $G_g=2$, respectively. Hence, the step sizes are fixed at
	$\tau=0.01$ and $\gamma_g=0.9$.  Finally, in Table~\ref{table:
		Conv_VI_schemes}, we examine the error when the number of
		generators increases. We terminate the regularized and
		extragradient scheme after $10000$ and $150000$ iterations and
		we present the error in solution iterates of learning function
		as well as the gap function associated with the true problem.
		Since the extragradient scheme is a constant steplength scheme,
		its performance appears to be significantly better than the
		regularized scheme but the latter does not necessitate knowledge
		of system parameters.
\begin{figure}[h]
\centering
  \centering
\includegraphics[scale=.38]{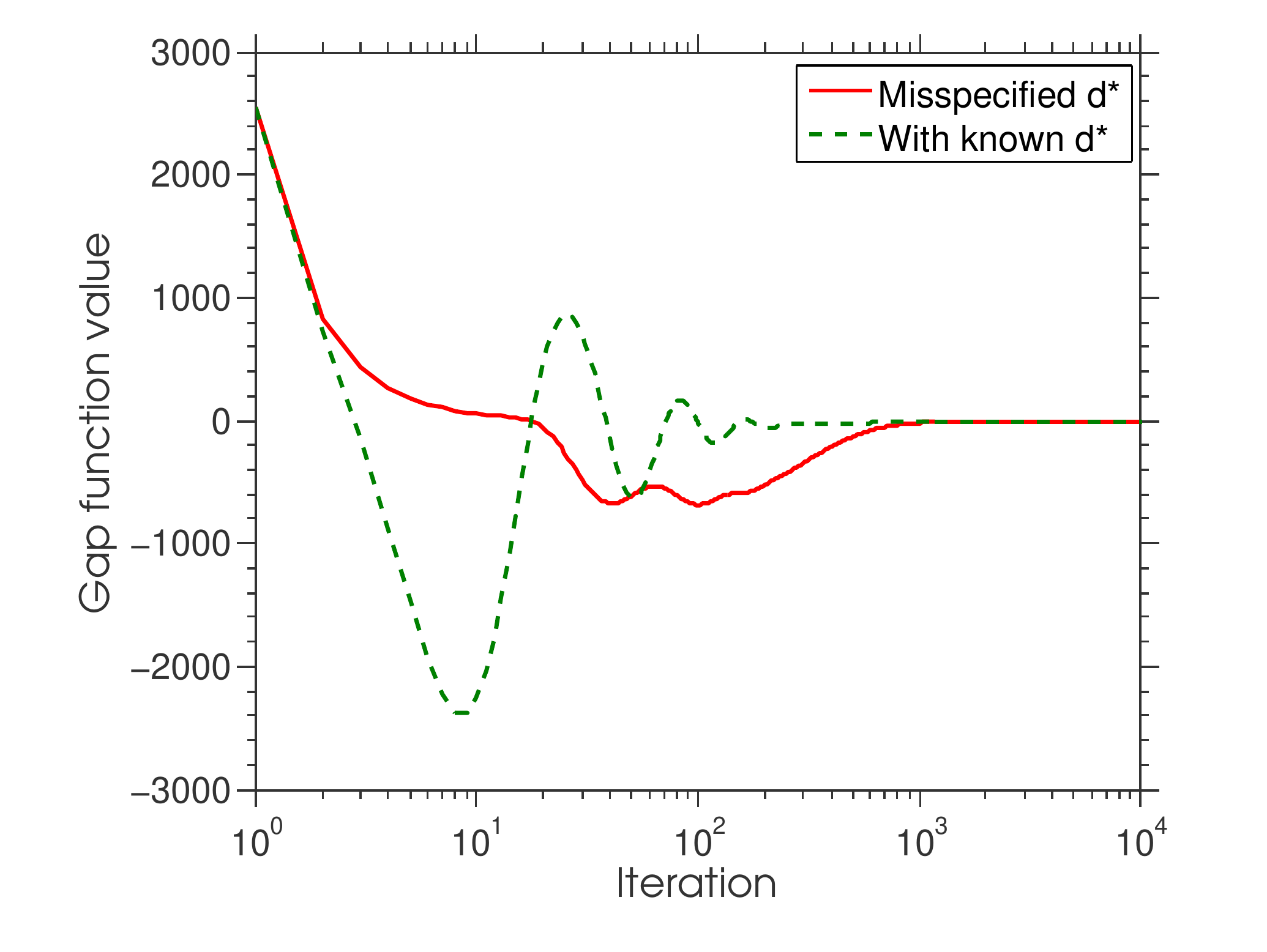}
  \centering
\includegraphics[scale=.4]{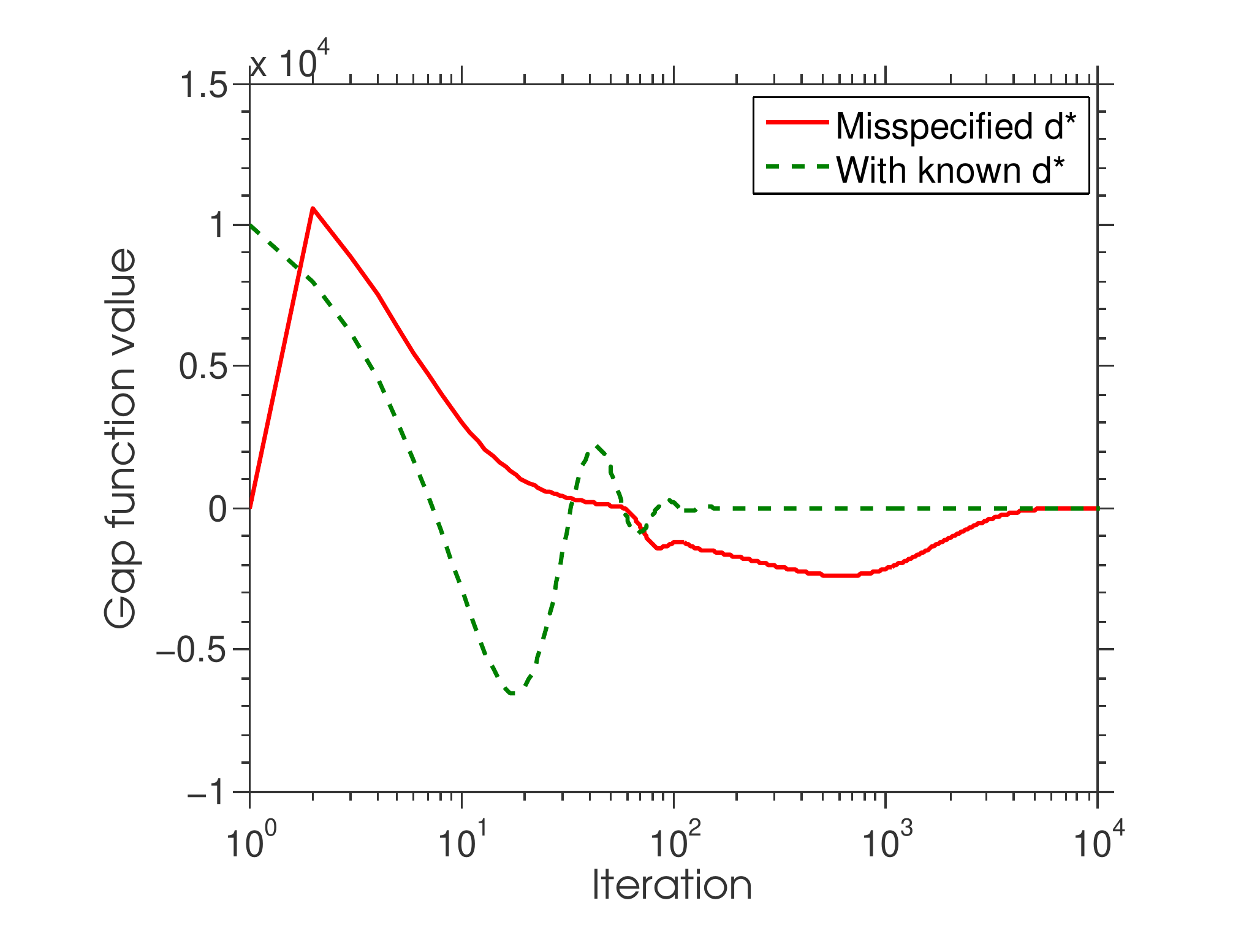}
\caption{Monotone VIs: Regularized schemes
	(l) and Extragradient schemes (r)}
\label{Fig:Con_VI}
\vspace{-0.1in}
\end{figure}

\begin{table}[htbp]
\scriptsize
\begin{center}
\begin{tabular}{|c|c|c|c|c|c|c|}
  \hline
  \multirow{2}{*}{  No. of generators} &  \multicolumn{2}{c|} {Extragradient scheme, $k=10000$}&
	  \multicolumn{2}{c|}{Regularization scheme, $k=150000$} \\\cline{2-5}
   &$\|d_k-d^*\|$ &G$(g_k;d^*)$&$\|d-d^*\|$ &G$(g_k;d^*)$\\
  \hline  \hline
  5& $4.5$$e$-$6$&$2.10$$e$-$5$&$2.7$$e$-$4$&$1.2$$e$-$3$\\\hline
10&$9.8$$e$-$6$&$4.36$$e$-$5$ &$3.8$$e$-$4$&$2.3$$e$-$3$\\\hline
15& $1.3$$e$-$5$&$6.5$$e$-$5$&$4.6$$e$-$4$&$3.4$$e$-$3$\\\hline
20&$1.8$$e$-$5$&$8.6$$e$-$5$&$5.3$$e$-$4$&$4.5$$e$-$3$\\\hline
\end{tabular}
\end{center}
 \caption{Convergence of extragradient and regularization schemes}
        \label{table: Conv_VI_schemes}
		\vspace{-0.2in}
\end{table}
\section{Concluding remarks} \label{sec:5}
The field of optimization algorithms has predominantly focused on the
resolution of optimization problems when the objective function and the
constraint set are known with certainty. However, in settings
complicated by large networked systems with streaming data, the
resulting optimization problems are often corrupted by a
misspecification, either in terms of the model or a prescribed
parameter. We focus on the second case and examine how one may resolve
this misspecification through a suitably defined learning process.  More
precisely, we formalize the setting as one where we have two coupled
computational problems; of these, the first is a misspecified
optimization problem while the second is a learning problem that arises
from having access to a learning data set, collected a priori. One
avenue for contending with such a problem is through an inherently
sequential approach that solves the learning problem and utilizes this
solution in subsequently solving the computational problem.
Unfortunately, unless accurate solutions of the learning problem are
available in finite time, it appears that sequential
approaches may not prove advisable. 

In this paper, we consider a simultaneous
approach that combines learning and computation via gradient-based
techniques. We make several contributions in this regard, broadly
categorized within the realm of misspecified convex optimization and
monotone variational inequality problems: (i) {\em Convex optimization
	problems:} First, in strongly convex regimes, it can be readily shown that constant
steplength gradient schemes admit global convergence properties. In
regimes where the strong convexity constants are unavailable, we
prove that suitably defined diminishing steplength schemes are also
shown to be convergent. Furthermore, we provide rate statements that demonstrate a
degradation the linear convergence rate, a consequence of
incorporating learning. Next, we consider problems where the
computational problem is merely convex and observe that both constant
steplength gradient
and subgradient methods see no change in the overall convergence rate
but instead display  a similar modification in their
rates given by $\|\theta_0-\theta^*\|{\cal O}(q_g^K +
		1/K)$.  This term is scaled by the initial misspecification in
$\theta$ and comprises of two terms, the first being a term that emerges
from learning the true $\theta^*$ and decays
to zero at a geometric rate while the second is an interaction term that
takes its rate from the averaging structure. When both the computation
and the learning problems are assumed to be merely convex with an
additional weak sharpness assumption on the learning problem, both
constant steplength and diminshing steplength statements may be
provided; {(ii) \em Variational inequality problems:} In the context of
monotone variational inequality problems, we present two sets of
techniques. Of these, the first is a constant steplength extragradient
scheme in which the steplength bound is modified to incorporate the
initial misspecification, given by $\|\theta_0-\theta^*\|$. Our second
scheme develops an iterative (Tikhonov) regularized scheme that does
rely on problem parameters and allows for recovery of the least norm
solution of the misspecified variational inequality problem.  Finally,
preliminary numerical tests support the theoretical findings and
remarkably the empirical convergence rates show a significant
superiority to theoretical bounds, suggesting that improvements may be
available.

Yet much remains to be understood about the realm of such techniques,
For instance, to what extent does the introduction of learning affect
the convergence rate in gradient methods as arising from Nesterov-type
acceleration techniques? Furthermore, can be develop analogous rate
statements for proximal and Lagrangian schemes and quantify the impacts
from learning? Finally, can we extend this framework to other
computational problems such as in the solution of Markov decision-making
problems (MDPs)?
\bibliography{ref}
\bibliographystyle{IEEEtran}
\end{document}